\documentclass[11pt]{amsart}
\usepackage{amsmath,amssymb,bm,a4wide}
\newtheorem{theorem}{Theorem}[section]
\newtheorem{lemma}[theorem]{Lemma}
\newtheorem{proposition}[theorem]{Proposition}
\newtheorem{corollary}[theorem]{Corollary}
\newtheorem{claim}{Claim}[section]
\theoremstyle{remark}
\newtheorem{remark}{Remark}[section]
\numberwithin{equation}{section}
\newcommand{\R}{\mathbb{R}}

\newcommand{\N}{\mathbb{N}}
\newcommand{\Z}{\mathbb{Z}}
\newcommand{\T}{\mathbb{T}}
\newcommand{\la}{\langle}
\newcommand{\ra}{\rangle}
\newcommand{\pd}{\partial}

\newcommand{\mgamma}{\bm{\gamma}}
\newcommand{\mk}{\bm{k}}
\newcommand{\eps}{\varepsilon}
\newcommand{\wR}{\widetilde{R}}
\newcommand{\pai}{\psi_{a,i}}

\newcommand{\tpai}{\tilde{\psi}_{a,i}}
\newcommand{\dc}{\delta c}
\newcommand{\dU}{\delta U}
\newcommand{\du}{\delta u}
\newcommand{\dv}{\delta v}
\newcommand{\dw}{\delta w}
\newcommand{\dg}{\delta\gamma}

\DeclareMathOperator{\sech}{sech}

\DeclareMathOperator{\diag}{diag}

\DeclareMathOperator{\spann}{span}

\begin{document}
\title{N-soliton states of the FPU lattices}
\author{Tetsu Mizumachi}
\address{Faculty of Mathematics, Kyushu University,
Fukuoka 819-0395 Japan}
\email{mizumati@math.kyushu-u.ac.jp}
\begin{abstract}
In this paper, we prove existence and uniqueness of solutions to
the Fermi Pasta Ulam lattice equation that converge to a sum of 
co-propagating $N$ solitary waves as $t\to\infty$ using linear stability
property of multi-soliton like solutions in an exponentially weighted space
proved by Mizumachi \cite{Mi2}. Counter-propagating
two soliton states have been studied by [Hoffman and Wayne,
Asymptotic two-soliton solutions in the Fermi-Pasta-Ulam model,
J. Dynam. Differential Equations  21  (2009), 343--351].
\end{abstract}
\maketitle
\section{Introduction}
\label{sec:intro}
In this paper, we prove existence and uniqueness of $N$-soliton like
solutions to FPU lattices
\begin{equation}
  \label{eq:lattice}
\ddot{q}(t,n)=V'(q(t,n+1)-q(t,n))-V'(q(t,n)-q(t,n-1))
\quad\text{for $\in\R$ and $n\in\Z$}  
\end{equation}
which models an infinite chain of anharmonic oscillators with
nearest-neighbor interaction potential $V$.
Making use of the change of variables
$$p(t,n)=\dot{q}(t,n)\quad r(t,n)=q(t,n+1)-q(t,n),$$
we can translate \eqref{eq:lattice} into a Hamiltonian system
\begin{equation}
  \label{eq:FPU}
  \frac{du}{dt}=JH'(u),
\end{equation}
where $J=\begin{pmatrix} 0 & e^\pd-1 \\ 1-e^{-\pd} & 0
\end{pmatrix}$, $e^{\pm\pd}$ are the shift operators defined
by $(e^{\pm\pd})f(n)=f(n\pm1)$ and
\begin{gather*}
H(u(t))=\sum_{n\in\Z}\left(\frac12p(t,n)^2+V(r(t,n))\right)
\quad\text{(Hamiltonian).}
\end{gather*}
\par
FPU lattices have solitary waves due to a balance between nonlinearity and
dispersion caused by discreteness of a spatial variable.
Friesecke and Wattis \cite{FW} prove that \eqref{eq:FPU}
has two parameter family of solitary wave solutions
$\{u_c(n-ct-\gamma) : c\in(-\infty,-1)\cup(1,\infty),\,\gamma\in\R\}$,
where $u_c=\begin{pmatrix}  r_c\\ p_c\end{pmatrix}$ is a solution of
\begin{equation}
  \label{eq:boundst}
c\pd_xu_c+JH'(u_c)=0.
\end{equation}
\par
In \cite{FP2,FP3,FP4}, Friesecke and Pego prove stability of 1-soliton
solutions of FPU lattices in an exponentially weighted space 
that are biased in the direction of motion of solitary waves. 
They utilize the fact the main solitary wave of a solution to \eqref{eq:FPU}
outruns from the other part of solutions.  Mizumachi \cite{Mi1} has shown that
1-solitons of \eqref{eq:FPU} are stable to perturbation in the energy class.
\par
If $V(r)=a(e^{br}-1-br)$ and $ab>0$, then  \eqref{eq:FPU} is an integrable
system so-called Toda lattice  and has explicit $N$-soliton solutions
(see \cite{To}).  That is, Toda lattice has a family of solutions which
converge to a sum of N solitary waves as $t\to\pm\infty$ and
solitary waves do not change their shape by collision.
However, it is not obvious whether multi-soliton like solutions can exist
stably in the non-integrable case.
\par
Existence and uniqueness of asymptotic $N$-soliton states of generalized 
KdV equations has been proved by Martel \cite{Ma} (see also \cite{MM1}).
His idea is to use stability theory of multi-solitons by Martel-Merle-Tsai
\cite{MMT} and monotonicity properties of localized  $L^2$ norms to ensure
uniqueness of the asymptotic $N$-soliton states. 
Recently, Martel and Merle \cite{MM2,MM3} prove that nonlinear interaction
between solitary waves causes defect of solitary waves in a setting different
from nearly integrable cases (see e.g. Hiraoka and Kodama \cite{HK}).
\par
One of the difference between FPU lattices and the KdV equation is that a
solitary wave of FPU lattices cannot be characterized as a critical point
of a conserved quantity due to the lack of infinitesimal invariance of the
spatial variable. Developing the method of \cite{FP1,FP2,FP3,FP4},
Hoffman and Wayne \cite{Hoff-Way} studied stability and head-on collision of
$2$-soliton states  waves propagating to the opposite direction.
They also prove the existence of solutions that converge to a sum of
counter-propagating solitary waves (see \cite{Hoff-Way2}).
If solitary waves move to the same direction,
the interaction through their tales are effective for a longer period
and we cannot derive strong linear stability of co-propagating multi-solitons
from \cite{FP3,FP4} as was done by Hoffman and Wayne
\cite{Hoff-Way, Hoff-Way2}. Recently, Mizumachi \cite{Mi2} has proved
stability of co-propagating $N$-soliton like solutions.

In this paper, we prove existence and uniqueness of solutions which converge
to a sum of $N$ solitary waves moving to the same direction
replacing the variational argument of Martel \cite{Ma} by the strong linear
stability property of multi-solitons in exponentially 
weighted spaces proved in \cite{Mi2}.
\par
Now, let us introduce our result.
\begin{theorem}
  \label{thm:1}
Suppose 
\begin{equation}
  \label{eq:H1}
V\in C^\infty(\R;\R),\quad V(0)=V'(0)=0, \quad V''(0)=1, \quad
 V'''(0)=\tfrac16.
\tag{H1}
\end{equation}
Let $k_{N}>\cdots>k_1>0$ and $c_{i,+}=1+\frac{(k_i\eps)^2}{6}$ for $1\le i\le N$.
There exists a positive number $\eps_0$ such that for any 
$\eps\in(0,\eps_0)$ and $\gamma_{i,+}\in\R$ $(1\le i\le N)$,
there exists a unique solution $u(t)$ of \eqref{eq:FPU} satisfying
\begin{equation}
  \label{eq:convergence}
\lim_{t\to\infty}\left\|u(\cdot,t)
-\sum_{i=1}^Nu_{c_{i,+}}(\cdot-c_{i,+}t-\gamma_{i,+})\right\|_{l^2}=0.
\end{equation}
\par
Furthermore, there exists a $\beta>0$ such that
\begin{equation}
  \label{eq:expconvergence}
\left\|u(\cdot,t)-\sum_{i=1}^Nu_{c_{i,+}}(\cdot-c_{i,+}t-\gamma_{i,+})
\right\|_{l^2} =O(\eps^{\frac32}e^{-\beta\eps h(t)})
\quad\text{as $t\to \infty$,}  
\end{equation}
where $h(t)=\min\{(c_{j,+}-c_{j-1,+})t+\gamma_{j,+}-\gamma_{j-1,+}
\,|\,2\le j\le N\}$.
\end{theorem}

 Our plan of the present paper is as follows:
In Section \ref{subsec:2.1}, we will show uniform boundedness of a
sequence $\{u_n(t)\}_{n\in\N}$, where 
$u_n(t)$ is a solution of \eqref{eq:FPU} which equals to a sum of $N$
solitary waves at $t=n$. Since the interaction between co-propagating solitary
waves are not strongly localized in space as counter-propagating $2$-solitons,
co-propagating solitary waves cannot be expected to be bounded in  $e^{-a|n|}l^2$
and the compactness argument of \cite{Hoff-Way2} does not work.
In Section \ref{sec:existence}, we will use monotonicity of localized norms
of $u_m(t)-u_n(t)$ to prove $\{u_n\}_{n=1}^\infty$ is a Cauchy sequence in $l^2$.
Using estimates obtained in Section \ref{sec:existence},
we prove uniqueness of solutions that converges to an 
$N$-soliton state in Section \ref{sec:uniqueness}. 
\par
In this paper, we will use the following properties of solitary wave
solutions proved by \cite{FP1}.
\begin{itemize}
\item[(P1)]
Let $c_*>1$ be a constant sufficiently close to $1$.
For any $c\in(1,c_*]$, there exists a unique single hump solution $u_c$
of \eqref{eq:boundst} in $l^2$ up to translation in $x$.
\item[(P2)] There exists an open interval $I$ such that 
$V''(r)>0$ for every $r\in I$ and that
$\overline{\{r_c(x):x\in\R\}}\subset I$ for every $c\in(1,1+c_*]$.
\item[(P3)] 
The solitary wave energy $H(u_c)$ satisfies
$dH(u_c)/dc\ne 0$ for $c\in(1,c_*]$.
\item[(P4)]
As $c$ tends to $1$, a shape of solitary wave solution becomes similar to
that of a KdV $1$-soliton. .More precisely, 
\begin{align*}
\sum_{j=0}^2\eps^j\left\|\pd_\varepsilon^j\left(
\eps^{-2}r_c\left(\frac{\cdot}{\eps}\right)-\sech^2x\right)
\right\|_{H^5(\R;e^{2a|x|}dx)}=O(\eps^{2})
\quad\text{for $a\in[0,2)$.}
\end{align*}
\end{itemize}
\par
Finally, let us introduce several notations.
Let
\begin{align*}
& \|u\|_{X_k^n(t)}=\| e^{-k_1\eps(\cdot-x_k^n(t))}u\|_{l^2},
\quad \|u\|_{X_k^n(t)^*}=\| e^{k_1\eps(\cdot-x_k^n(t))}u\|_{l^2},
\\ &
\|u\|_{W_k^n(t)}=\sum_{i=1}^k\|e^{-k_1\eps|\cdot-x_i^n(t)|}u\|_{l^2},
\quad \|u\|_{W_k^n(t)^*}=\min_{i=1}^k \|e^{k_1\eps|\cdot-x_i^n(t)|}u\|_{l^2},
\\ &
\|u\|_{\widetilde{W}_k^n(t)}=\sum_{i=1}^k\|e^{-k_1\eps|\cdot-x_i^n(t)|}u\|_{l^1},
\quad \|u\|_{\widetilde{W}_k^n(t)^*}
=\min_{i=1}^k \|e^{k_1\eps|\cdot-x_i^n(t)|}u\|_{l^\infty}.
\end{align*}
We abbreviate $W_N^n(t)$ as $W^n(t)$.
For Banach spaces $X$ and $Y$, we denote by $B(X,Y)$ the space of 
all linear continuous operators from $X$ to $Y$ and abbreviate
$B(X,X)$ as $B(X)$.
We use $a\lesssim b$ and $a=O(b)$ to mean that there exists a positive constant
such that $a\le Cb$.

\section{Approximation of N-soliton like states}
\label{subsec:2.1}
Let $x_{i,+}(t)=c_{i,+}t+\gamma_{i,+}$ and let
$u^n(t)$ be a solution to
\begin{equation}
  \label{eq:un}
  \left\{    \begin{aligned}
& \pd_tu^n=JH'(u^n),\\
& u^n(n)=\sum_{i=1}^Nu_{c_{i,+}}(\cdot-x_{i,+}(n)).
    \end{aligned}\right.
\end{equation}
In this section, we will prove uniform boundedness of $u^n(t)$
for $t\in [T,n]$.
To apply the strong linear stability property in exponentially weighted
spaces (Lemma \ref{lem:linearstability}), we will decompose $u^n$ into a sum of 
solitary waves and remainder parts following the idea of \cite{Mi2}.
\par
In view of the proof of \cite[Proposition 1]{Mi1},
we have $u^n(t)\in l^2_a\cap l^2_{-a}$ for any 
$0\le a<2\kappa(c_1)$, where $\kappa(c_1)$ is a solution to
$\frac{\sinh\kappa}{\kappa}=c_1$.
Let $P_k^n$ and $Q_k^n$ be projections defined on $l^2_{-a}$
$(0<a<2\kappa(c_1^n))$ such that
$P_k^n+Q_k^n=I$ and
\begin{align*}
\operatorname{Range}P_k^n &=\spann
\{\pd_xu_{c_i^n}(\cdot-x_i^n), \pd_cu_{c_i^n}(\cdot-x_i^n)
\,|\, 1\le i\le k\},\\
{}^\perp\operatorname{Range}Q_k^n &=\spann\{J^{-1}\pd_xu_{c_i^n}(\cdot-x_i^n),
J^{-1}\pd_cu_{c_i^n}(\cdot-x_i^n)\,|\,1\le i\le k\},
\end{align*}
where $J$ is a bounded inverse of $J$ on $l^2_a$ $(a>0)$. That is,
\begin{equation}\label{eq:J-1}
J^{-1}= -\begin{pmatrix} 0 & \sum_{k=1}^\infty e^{k\pd}
\\ \sum_{k=0}^\infty e^{k\pd} & 0 \end{pmatrix}
\end{equation}
Let $v_k^n(t)$ $(1\le k\le N)$ be a solution of 
  \begin{equation}
    \label{eq:vk}
    \left\{\begin{aligned}
& \pd_tv_k^n=JH''(U_k^n)v_k^n+l_k^n+Q_k^nJR_k^n,\\
& v_k^n(n)=0,
      \end{aligned}\right.
  \end{equation}
where $w_0^n=0$, $w_k^n=\sum_{i=1}^kv_i^n$,
\begin{align}
\label{eq:Ukdef}
U_0^n(t)=&0,\quad U_k^n(t)=\sum_{i=1}^k u_{c_i^n(t)}(\cdot-x_i^n(t)),\\
\label{eq:Rkdef}
R_k^n=&  H'(U_k^n+w_k^n)-H'(U_{k-1}^n+w_{k-1}^n)-H'(u_{c_k^n}(\cdot-x_k^n))
-H''(U_k^n)v_k^n,
\\
\label{eq:lkdef}
l_k^n=& \sum_{i=1}^k\left(\alpha_{ik}^n\pd_cu_{c_i^n}(\cdot-x_i^n)
+\beta_{ik}^n\pd_xu_{c_i^n}(\cdot-x_i^n)\right).
\end{align}
We will choose $\alpha_{ik}^n(t)$, $\beta_{ik}^n(t)$, $x_i^n(t)$ and $c_i^n(t)$
so that $v_k^n$ satisfies the symplectic orthogonality condition
\eqref{eq:orthv2k} and that $U_N^n+w_N^n$ is a solution of \eqref{eq:un}.
\par

Let
\begin{gather*}
\mathcal{A}_k^n=\begin{pmatrix}\mathcal{A}_{i,j}^n
\end{pmatrix}_{\substack{i=1,\cdots, k\downarrow\\ j=1,\cdots, k\rightarrow}},
\quad 
\widetilde{\mathcal{A}}_k^n=\begin{pmatrix}\mathcal{A}_{i,j}^n
\end{pmatrix}_{\substack{i=1,\cdots, N\downarrow\\ j=1,\cdots, k\rightarrow}},\quad
\delta\mathcal{A}_k^n=\diag(\delta\mathcal{A}_{ik}^n)_{1\le i\le k},
\\
\mathcal{A}_{i,j}^n=
 \begin{pmatrix}
\eps^{-1}\la \pd_cu_{c_j^n}(\cdot-x_j^n),
J^{-1}\pd_xu_{c_i^n}(\cdot-x_i^n)\ra 
& \eps^{-4}\la \pd_xu_{c_j^n}(\cdot-x_j^n),
J^{-1}\pd_xu_{c_i}(\cdot-x_i^n)\ra
 \\  \eps^2\la \pd_cu_{c_j^n}(\cdot-x_j^n),
J^{-1}\pd_cu_{c_i^n}(\cdot-x_i^n)\ra
 & \eps^{-1}\la \pd_xu_{c_j^n}(\cdot-x_j^n),
J^{-1}\pd_cu_{c_i}(\cdot-x_i^n)\ra
 \end{pmatrix},
\\
\delta\mathcal{A}_{ik}^n=
\begin{pmatrix} \eps^{-1}\la v_k^n ,
J^{-1}\pd_c\pd_xu_{c_i^n}(\cdot-x_i^n)\ra
& \eps^{-4}\la v_k^n, J^{-1}\pd_x^2u_{c_i^n}(\cdot-x_i^n)\ra \\ 
\eps^2\la v_k^n, J^{-1}\pd_c^2u_{c_i}(\cdot-x_i^n)\ra
& \eps^{-1}\la v_k^n, J^{-1}\pd_c\pd_xu_{c_i^n}(\cdot-x_i^n)\ra
\end{pmatrix}.
\end{gather*}

\begin{remark}
We  will use exponentially weighted spaces $X_k^n(t)^*$
to prove Theorem \ref{thm:1}.
The weight function of $X_k^n(t)^*$ is monotone decreasing whereas the weight
function of $X_k(t)$ used in \cite{Mi2} is monotone increasing. 
This is because we will solve \eqref{eq:FPU} backward in time.
The order of the decomposition of nonlinearity of \eqref{eq:FPU}
in \eqref{eq:vk} is the opposite from that in \cite[(2.6)]{Mi2}.
\end{remark}

Let
\begin{align*}
& \delta R_{1k}^n(i)=
\begin{pmatrix}\eps^{-4}\la R_k^n,\pd_xu_{c_i^n}(\cdot-x_i^n)\ra
\\ \eps^{-1}\la R_k^n, \pd_cu_{c_i^n}(\cdot-x_i^n)\ra\end{pmatrix},
\\
&  \delta R_{2k}^n(i)=
\begin{pmatrix} \eps^{-4}\la v_k^n,
(H''(U_k^n)-H''(u_{c_i^n}(\cdot-x_i^n)))\pd_xu_{c_i^n}(\cdot-x_i^n)\ra
\\ \eps^{-1}\la v_k^n,
(H''(U_k^n)-H''(u_{c_i^n}(\cdot-x_i^n)))\pd_cu_{c_i^n}(\cdot-x_i^n)\ra
\end{pmatrix},
\\ & 
\delta R_{1k}^n=
\begin{pmatrix}\delta R_{1k}^n(i)\end{pmatrix}_{i=1,\cdots,k\downarrow},\quad
\delta R_{2k}^n=
\begin{pmatrix}\delta R_{2k}^n(i)\end{pmatrix}_{i=1,\cdots,k\downarrow},
\\ &
\delta\widetilde{\mathcal{A}}_k^n
=\widetilde{\mathcal{A}}_k^n(\mathcal{A}_k^n)^{-1}\delta\mathcal{A}_k^n,
\quad 
\delta\wR_{1k}^n=\widetilde{\mathcal{A}}_k^n(\mathcal{A}_k^n)^{-1}\delta R_{1k},
\quad
\delta\wR_{2k}^n=\widetilde{\mathcal{A}}_k^n(\mathcal{A}_k^n)^{-1}\delta R_{2k}.
\end{align*}
\begin{lemma}
  \label{lem:modulation}
Let $v_i^n$ $(1\le i\le N)$ be solutions of \eqref{eq:vk}. Suppose
\begin{gather}
\label{eq:xcinit}
x_i^n(n)=x_{i,+}(n),\quad c_i^n(n)=c_{i,+},\\
  \label{eq:modeq1}
\left(\mathcal{A}_N^n-\sum_{k=1}^N\delta\widetilde{\mathcal{A}}_k^n\right)
\begin{pmatrix}\eps^{-3}\dot{c}_i^n\\ c_i^n-\dot{x}_i^n
\end{pmatrix}_{i=1,\cdots,N\downarrow}
+\sum_{k=1}^N(\delta\wR_{1k}^n+\delta\wR_{2k}^n)=0,
\end{gather}
and that $\alpha_{ij}^n$ and $\beta_{ij}^n$ ($1\le i\le k\le N$) are
$C^1$-functions satisfying
\begin{equation}
  \label{eq:orthvk3}
\begin{split}
& \mathcal{A}_k^n\begin{pmatrix}
\eps^{-3}\alpha_{ik}^n(t)\\ \beta_{ik}^n(t) \end{pmatrix}_{i=1,\cdots,k\downarrow}
\\ =& \delta R_{2k}^n +\delta\mathcal{A}_k^nE_{k,N}
(\mathcal{A}_N^n-\sum_{k=1}^N\delta\widetilde{\mathcal{A}}_k^n)^{-1}
\sum_{k=1}^N(\delta\wR_{1k}^n+\delta\wR_{2k}^n),
\end{split}
\end{equation}
where $E_{k,N}=(\delta_{ij})_{\substack{i=1,\cdots,2k\downarrow\\j=1,\cdots,2N\rightarrow}}$.
Then $v_k^n$ $(1\le k\le N)$ satisfies
\begin{equation}
\label{eq:orthv2k}
\la v_k^n(t),J^{-1}\pd_xu_{c_i^n(t)}(\cdot-x_i^n(t))\ra
=\la v_k^n(t),J^{-1}\pd_cu_{c_i^n(t)}(\cdot-x_i^n(t))\ra=0
\end{equation}
for $1\le i\le k$ and $t\in\R$.
Moreover, $U_N^n+w_N^n$ is a solution of \eqref{eq:un}.
\end{lemma}
\begin{proof}
Due to the symmetry of \eqref{eq:boundst} with respect to $x$ and $c$,
we have
\begin{equation}
  \label{eq:secularmode}
c\pd_x^2u_{c}+JH''(u_{c})\pd_xu_{c}=0,\quad
c\pd_c\pd_xu_{c}+JH''(u_{c})\pd_cu_{c}=-\pd_xu_c.
\end{equation}
By \eqref{eq:modeq1} and \eqref{eq:orthvk3},
\begin{equation}
  \label{eq:modeq1'}
\mathcal{A}_k^n(t)\begin{pmatrix}
\eps^{-3}\alpha_{jk}^n(t)\\\beta_{jk}^n(t)
\end{pmatrix}_{1\le j\le k\downarrow}
=\delta R_{2k}^n(t)-\delta\mathcal{A}_k^n(t)
\begin{pmatrix}
\eps^{-3}\dot{c}_j^n(t)\\ c_j^n(t)-\dot{x}_j^n(t)
\end{pmatrix}_{1\le j\le k\downarrow}.  
\end{equation}
Differentiating \eqref{eq:orthv2k} with respect to $t$ and substituting
\eqref{eq:vk} and \eqref{eq:modeq1'} into the resulting equations,
we have for $1\le i\le k$, 
\begin{equation}
  \label{eq:modeqa}
  \begin{split}
&  \frac{d}{dt}\la v_k^n,J^{-1}\pd_xu_{c_j^n}(\cdot-x_j^n(t))\ra
\\=&
\la l_k^n, J^{-1}\pd_xu_{c_j^n}\ra
-\la v_k^n,(H''(U_k^n)-H''(u_{c_j^n}))\pd_xu_{c_j^n}\ra \\ &
-(\dot{x}_j^n-c_j^n)\la v_k^n,J^{-1}\pd_x^2u_{c_j^n}\ra
+\dot{c}_j^n\la v_k^n,J^{-1}\pd_c\pd_xu_{c_j^n}\ra=0,    
  \end{split}
\end{equation}
and
\begin{equation}
  \label{eq:modeqb}
  \begin{split}
&  \frac{d}{dt}\la v_k^n,J^{-1}\pd_cu_{c_j^n}(\cdot-x_j^n(t))\ra
\\=&
\la l_k^n, J^{-1}\pd_cu_{c_j^n}\ra
-\la v_k^n,(H''(U_k^n)-H''(u_{c_j^n}))\pd_cu_{c_j^n}\ra
+\la v_k^n,J^{-1}\pd_xu_{c_j^n}\ra
\\ &
-(\dot{x}_j^n-c_j^n)\la v_k^n,J^{-1}\pd_c\pd_xu_{c_j^n}\ra
+\dot{c}_j^n\la v_k^n,J^{-1}\pd_c^2u_{c_j^n}\ra
\\=& \la v_k^n,J^{-1}\pd_xu_{c_j^n}\ra.
  \end{split}
\end{equation}
Since $v_k^n(n)=0$, we have  \eqref{eq:orthv2k}.
\par
Finally, we will prove $u^n=U_N^n+w_N^n$.
It follows from \eqref{eq:vk} and the definition of $U_k$ that 
for $k=1,\cdots, N$,
\begin{equation}
  \label{eq:ukwk}
  \pd_t(U_k^n+w_k^n)=JH'(U_k^n+w_k^n)+\tilde{l}_k^n+ \sum_{i=1}^k(l_i^n-P_i^nJR_i^n),
\end{equation}
where $\tilde{l}_k^n=
\sum_{i=1}^k\{\dot{c}_i^n\pd_cu_{c_i^n}-(\dot{x}_i^n-c_i^n)\pd_xu_{c_i^n}\}$.
Hence it suffices to show
\begin{equation}
  \label{eq:diffpro}
 \tilde{l}_N^n+\sum_{i=1}^N\left(l_i^n-P_i^nJR_i^n\right)=0.
\end{equation}
The projection $P_k^n(t)$ can be written as 
\begin{equation}
  \label{eq:pkdef}
P_k^n(t)f=(\eps^3\pd_cu_{c_j^n},\pd_xu_{c_j^n})_{j=1,\cdots,k\rightarrow}
(\mathcal{A}_k^n)^{-1}
\begin{pmatrix}  \eps^{-4}\la f,J^{-1}\pd_xu_{c_i^n}\ra
\\ \eps^{-1}\la f,J^{-1}\pd_cu_{c_i^n}\ra
\end{pmatrix}_{i=1,\cdots,k\downarrow}.
\end{equation}
By \eqref{eq:modeq1}, \eqref{eq:modeq1'} and \eqref{eq:pkdef},
\begin{align*}
&  \begin{pmatrix}
\eps^{-4}\sum_{k=1}^N\la l_k^n-P_k^nJR_k^n, J^{-1}\pd_xu_{c_i^n}\ra
\\ \eps^{-1}\sum_{k=1}^N\la l_k^n-P_k^nJR_k^n, J^{-1}\pd_cu_{c_i^n}\ra
  \end{pmatrix}_{i=1,\cdots,N\downarrow}
\\=&
\sum_{k=1}^N\left\{\widetilde{\mathcal{A}}_k^n
\begin{pmatrix}\eps^{-3}\alpha_{ik}^n \\ \beta_{ik}^n 
\end{pmatrix}_{i=1,\cdots,k\downarrow} +\delta\wR_{1k}^n\right\}
\\=& \sum_{k=1}^N\left\{\delta\wR_{1k}^n+\delta \wR_{2k}^n
-\delta\widetilde{\mathcal{A}}_k^n
\begin{pmatrix}\eps^{-3}\dot{c}_i^n \\ c_i^n-\dot{x}_i^n
\end{pmatrix}_{i=1,\cdots,N\downarrow}\right\}
\\=& -
\begin{pmatrix}
\eps^{-4}\la \tilde{l}_N^n, J^{-1}\pd_xu_{c_i^n}\ra \\
\eps^{-1}\la \tilde{l}_N^n, J^{-1}\pd_cu_{c_i^n}\ra
\end{pmatrix}_{i=1,\cdots,N\downarrow}.
\end{align*}
Thus we have \eqref{eq:diffpro} since the left hand side of
\eqref{eq:diffpro} belongs to $\operatorname{Range}(P_N^n(t))$.
This completes the proof of Lemma \ref{lem:modulation}.
\end{proof}

\begin{lemma}
  \label{lem:parametersupb}
Let $c_{i,+}$ be as in Theorem \ref{thm:1} and let $v_k^n(t)$, $c_k^n(t)$
and $x_k^n(t)$ $(1\le k\le N)$ be as in Lemma \ref{lem:modulation}.
Then there exist positive numbers $\eps_0$, $\delta$, $L_0$, $C$ and $n_0\in\N$
satisfying the following: Suppose $\eps\in(0,\eps_0)$ and that for a $T>0$,
\begin{gather*} \sup_{n\ge n_0}\sup_{t\in[T,n]}\sum_{k=1}^N
\left(|c_k^n(t)-c_{k,+}|+\|v_k^n(t)\|_{X_k^n(t)^*}+\|v_k^n(t)\|_{l^2}\right)
\le \delta\eps^2,
\\ L:=\inf_{n\ge n_0}\inf_{t\in[T,n]}\min_{2\le k\le N}
\eps(x_k^n(t)-x_{k-1}^n(t))\ge L_0.
 \end{gather*}
Then for every $n\ge n_0$, $T\le t\le n$ and $1\le k\le N$,
\begin{gather}
\label{eq:xcupb}
\eps^{-3}|\dot{c}_k^n|+|\dot{x}_k^n-c_k^n|
\le C \eps^{\frac12}\sum_{i=1}^N\|v_i^n\|_{W^n(t)}+\eps^2e^{-k_1\eps d(t)},
\\ \label{eq:alphabetaupb}
\begin{split}
& \sum_{j=1}^k e^{k_1\eps(x_k^n(t)-x_j^n(t))}
(\eps^{-3}|\alpha_{jk}^n(t)|+|\beta_{jk}^n(t)|) \\  \le & C 
\eps^{-1}\|v_k^n\|_{X_k^n(t)^*}
\left(\sum_{i=1}^N\|v_i^n\|_{W^n(t)}+\eps^{\frac32}e^{-k_1\eps d(t)}\right).
\end{split}
\end{gather}
\end{lemma}
\begin{proof}
Lemma \ref{lem:parametersupb} can be proved in the same way as
\cite[Lemma 2.5]{Mi2}. Since $J^{-1}\pd_cu_{c_i^n}$ decays exponentially as
$n\to\infty$, it follows from Claim \ref{cl:j-1} that
\begin{equation}
\label{eq:Aij}
  \mathcal{A}_{ij}^n=\left\{\begin{aligned}
& B_1(c_i^n)\quad\text{if $i=j$,}\\
& O(e^{-k_1\eps(x_j^n-x_i^n)})\quad\text{if $i<j$,}\\
& B_2(c_i^n,c_j^n)+O(e^{-k_1\eps(x_i^n-x_j^n)})\quad\text{if $i>j$,}\\
\end{aligned}\right.
\end{equation}
where $\theta_1(c):=dH(u_c)/dc$, 
$\theta_2(c):=\la \pd_c p_{c},1\ra\la \pd_cr_{c},1\ra$,
\begin{align*}
& \theta_3(c_i,c_j):=\la \pd_cp_{c_i},1\ra \la \pd_c r_{c_j},1\ra
+\la \pd_cp_{c_j},1\ra\la \pd_cr_{c_i},1\ra,\\
& \sigma_3:=\begin{pmatrix}1 & 0 \\ 0 & -1\end{pmatrix},
\quad B_1(c):=-(c\eps)^{-1}\theta_1(c)\sigma_3-\eps^2\theta_2(c)
\begin{pmatrix}0 & 0\\ 1 & 0\end{pmatrix},
\\ & B_2(c_i,c_j):=-\eps^2\theta_3(c_i,c_j)
\begin{pmatrix}0 & 0\\ 1 & 0\end{pmatrix}.
\end{align*}
We remark that the dominant part of $\mathcal{A}_k$ is a lower triangular
matrix.
\par
The other part of the proof are exactly the same as that of
\cite[Lemma 2.5]{Mi2}.
\end{proof}

Now we will estimate energy norm of $v_k^n$ $(1\le k\le N)$.
\begin{lemma}
  \label{lem:speed-Hamiltonian}
Let $c_{i,+}$ be as in Theorem \ref{thm:1} and let $v_k^n(t)$, $c_k^n(t)$,
$x_k^n(t)$ $(1\le k\le N)$ and $T$ be as in Lemma \ref{lem:parametersupb}.
Then there exist $n_0\in\N$ and a positive constant $C$ such that for
every $n\ge n_0$ and $t\in[T,n]$,
\begin{equation}
  \label{eq:enorm}
\|v^n(t)\|_{l^2}^2\le C\left(\sum_{i=1}^N(\eps|c_i^n(t)-c_{i,+}|+
\eps^{\frac32}\|v_i^n\|_{W^n(t)})+\eps^3e^{-k_1\eps d(t)}\right),
\end{equation}
\begin{equation}
  \label{eq:locenorm}
  \begin{split}
\|v_k^n(t)\|_{l^2}^2 \le & C\sum_{i=1}^k\left(
\eps|c_i^n(t)-c_{i,+}|+\eps^{\frac32}\|v_i^n\|_{W_k^n(t)}\right)
\\ & +C\eps^3\left(\sum_{i=1}^N\|v_i^n\|_{L^2(t,n;W^n(s)\cap X_i^n(s)^*)}^2
+e^{-k_1\eps d(t)}\right),
  \end{split}
\end{equation}
where $d(t)=\sigma\eps^2(t-T)+\eps^{-1}L$ and
$\sigma=\frac12\eps^{-2}\min_{2\le i\le N}(c_{i,+}-c_{i-1,+})$.
\end{lemma}
\begin{proof}
  First we remark that
\begin{align*}
 x_{i+1}^n(t)-x_i^n(t)=& x_{i+1}^n(T)-x_i^n(T)
+\int_T^t(\dot{x}_{i+1}^n(s)-\dot{x}_i^n(s))ds
\\ \ge & \eps^{-1}L+\left(c_{i+1,+}-c_{i,+}-O(\delta\eps^2)\right)(t-T)
\\ \ge & d(t).
\end{align*}
\par

By (P2), there exists a positive constant $C'$
independent of $\eps$ such that
  \begin{align*}
dH:=& H(u^n(t))-\sum_{i=1}^NH(u_{c_{i,+}})
\\ = & H(U_N^n(t)+v^n(t))-\sum_{i=1}^NH(u_{c_{i,+}})
\\ = & I_1+I_2+\frac12\la H''(U_N^n)v^n,v^n\ra+O(\|v^n\|_{l^2}^3)
\\ \ge & C'\|v^n(t)\|_{l^2}^2+I_1+I_2,
  \end{align*}
where
$I_1=\la H'(U_N^n(t)),v^n(t)\ra$ and $I_2=H(U_N^n(t))-\sum_{i=1}^NH(u_{c_{i,+}}).$
It follows from \eqref{eq:orthvk3}, the fact that
$|H'(U_N^n(t))|\lesssim \sum_{i=1}^N |u_{c_i^n(t)}(\cdot-x_i^n(t))|$ and
Claims \ref{cl:ucsize}--\ref{cl:4} that 
$$|I_1| \le  \eps^{\frac32}\sum_{k=1}^{N-1}\|v_k^n\|_{W^n(t)}
+O\left(\eps^{\frac72}\sum_{k=1}^N\|v_k^n\|_{W^n(t)}\right),$$
\begin{align*}
  |I_2|\le &
\sum_{i=1}^N|H(u_{c_i^n(t)})-H(u_{c_{i,+}})|
+\left|H(U_N^n(t))-\sum_{i=1}^NH(u_{c_i^n(t)}(\cdot-x_i^n(t))\right|
\\ \lesssim &
\sum_{i=1}^N\theta_1(c_{i,+})|c_i^n(t)-c_{i,+}|
+\sum_{j\ne i}
\left\|u_{c_i^n(t)}(\cdot-x_i^n(t))u_{c_j^n(t)}(\cdot-x_j^n(t))\right\|_{l^1}
\\ \lesssim & 
\eps\sum_{i=1}^N|c_i^n(t)-c_{i,+}|+\eps^3e^{-k_1\eps d(t)}.
\end{align*}

Since $dH$ is independent of $t$,
\begin{align*}
|dH|=\left| H(U_N^n(n))-\sum_{1\le i\le N}H(u_{c_{i,+}})\right|
\lesssim  \eps^3e^{-k_1\eps h(n)}.
\end{align*}
Thus we prove \eqref{eq:enorm}.
\par

To prove \eqref{eq:locenorm}, it suffices to show
\begin{equation}
  \label{eq:w2knorm}
  \begin{split}
& \|w_k^n(t)\|_{l^2}^2 \lesssim  \eps\sum_{i=1}^k|c_i^n(t)-c_{i,+}|
+\eps^{\frac32}\sum_{i=1}^{k}\|v_i^n\|_{W_k^n(t)}
\\ & 
+\eps^3\left(\sum_{i=1}^N\|v_i^n\|_{L^2(t,n;W^n(s)\cap X_i^n(s)^*)}^2
+e^{-k_1\eps d(t)}\right).
  \end{split}
\end{equation}
Since  $J$ is skew-adjoint, it follows from \eqref{eq:ukwk} that
\begin{equation}
  \label{eq:ham21}
\begin{split}
\frac{d}{dt}H(U_k^n+w_k^n)= & \left\la H'(U_k^n+w_k^n),
\tilde{l}_k^n+\sum_{i=1}^k(l_i^n-P_i^nJR_i^n)\right\ra
= \sum_{i=1}^6II_i,
\end{split}
\end{equation}
where 
$U_{k,int}^n=H'(U_k^n)-\sum_{i=1}^k H'(u_{c_i^n})=
\sum_{\substack{1\le i,j\le k \\ i\ne j}}O\left(|u_{c_i^n}(\cdot-x_i^n)|
|u_{c_j^n}(\cdot-x_j^n)|\right)$ and
\begin{align*}
II_1=& \sum_{i=1}^k\la H'(U_k^n+w_k^n),l_i^n\ra,\quad
II_2= -\sum_{i,\,j=1}^k\la H'(u_{c_j^n}),P_i^nJR_i^n\ra,\\
II_3=& -\sum_{i=1}^k\la U_{k,int}^n ,P_i^nJR_i^n\ra,\quad
II_4= -\sum_{i=1}^k\la H'(U_k^n+w_k^n)-H'(U_k^n),P_i^nJR_i^n\ra,\\
II_5=& \sum_{j=1}^k \la H'(u_{c_j^n}), \tilde{l}_k^n\ra,\quad
II_6=\la H(U_k^n+w_k^n)-\sum_{j=1}^k H(u_{c_j^n}),\tilde{l}_k^n\ra.
\end{align*}
By Lemma \ref{lem:parametersupb} and the fact that
$\|H'(U_k^n+w_k^n)\|_{l^2}=O(\eps^{\frac32})$,
\begin{equation}
\label{eq:ham22}
|II_1|\lesssim 
\eps^3 \sum_{k=1}^{N}\|v_k^n\|_{X_k^n(t)^*\cap W^n(t)}^2+\eps^6e^{-k_1\eps d(t)}.
\end{equation}
Next, we will estimate $II_2$.
Let
\begin{align*}
& R_{k1}^n=H'(U_k^n+w_k^n)-H'(U_k^n+w_{k-1}^n)-H''(U_k^n)v_k^n,\\
& R_{k2}^n=H'(U_k^n)-H'(U_{k-1}^n)-H'(u_{c_k^n}),\\
& R_{k3}^n=H'(U_k^n+w_{k-1}^n)-H'(U_{k-1}^n+w_{k-1}^n)-H'(U_k^n)+H'(U_{k-1}^n).
\end{align*}
Then
\begin{equation}
\label{eq:rkn}
  \begin{split}
&  R_k^n=R_{k1}+R_{k2}^n+R_{k3}^n,\\
& |R_{k1}^n|\lesssim (|w_{k-1}^n|+|v_k^n|)|v_k^n|,\quad
|R_{k2}^n|\lesssim |u_{c_k^n}||U_{k-1}^n|,\\
& |R_{k3}|\lesssim  |u_{c_k^n}||w_{k-1}^n|.
  \end{split}
\end{equation}
As in \cite[Lemma 3.1]{Mi2}, we have
$(P_i^nJ)^*H'(u_{c_j^n})=c_j^n\pd_xu_{c_j^n}$ for $j\le i$ and
$(P_i^nJ)^*H'(u_{c_j^n})=O(\eps^3e^{-k_1\eps|x_j-x_i|})$ for $j>i$.
Hence it follows from \eqref{eq:rkn} that
\begin{equation}
  \label{eq:ham23}
  \begin{split}
II_2=& -\sum_{i=1}^k c_i^n\la R_{i3}^n,\pd_xu_{c_i^n}\ra
+O\left(\eps^3\sum_{i=1}^k\|v_i^n\|_{W^n(t)}^2+\eps^6e^{-k_1\eps d(t)}\right).
  \end{split}
\end{equation}
\par
Secondly, we will estimate $II_3$ and $II_4$.
Since
\begin{equation}
  \label{eq:pkupb}
\|P_k^nJ\|_{B(l^1)}+\|P_k^nJ\|_{B(W_k^n(t),W_k^n(t)^*)}+
\eps^{-\frac12}\|P_k^nJ\|_{B(\widetilde{W}_k(t),W_k^n(t)^*)}=O(\eps),
\end{equation}
it follows from Claim \ref{cl:intsize}, \eqref{eq:rkn} and the assumption of
Lemma \ref{lem:speed-Hamiltonian} that
\begin{align}
  \label{eq:ham24}
|II_3|\le & \|U_{k,int}^n\|_{l^2}\sum_{i=1}^k\|P_i^nJR_i^n\|_{l^2}
\lesssim  \eps^8e^{-k_1\eps d(t)},
\\ \label{eq:ham25}  
\begin{split}
|II_4|  \lesssim & \|w_k^n\|_{W^n(t)}
\left(\eps^3\sum_{i=1}^k\|v_i^n\|_{W^n(t)}+\eps^{\frac92}e^{-k_1\eps d(t)}\right).
\end{split}
\end{align}
By \eqref{eq:xcupb},
\begin{equation}
  \label{eq:ham26}
  \begin{split}
II_5=& \sum_{i=1}^k\left\{\theta_1(c_i^n)\dot{c}_i^n
+O(e^{-k_1\eps d(t)}(\eps|\dot{c}_i^n|+\eps^4|\dot{x}_i^n-c_i^n|))\right\}
\\ =& \sum_{i=1}^k\theta_1(c_i^n)\dot{c}_i^n+O\left(\eps^6e^{-k_1\eps d(t)}\right),
  \end{split}
\end{equation}
\begin{equation}
  \label{eq:ham27}
  \begin{split}
|II_6|\lesssim & (\|U_{k,int}^n\|_{W^n(t)}+\|w_k^n\|_{W^n(t)})
\|\tilde{l}_k^n\|_{W^n(t)^*}
\\ \lesssim & 
(\eps^{\frac72}e^{-k_1\eps d(t)}+\|w_k^n\|_{W^n(t)})
\sum_{i=1}^k (\eps^{-\frac12}|\dot{c}_i^n|+\eps^{\frac52}|\dot{x}_i^n-c_i^n|)
\\ \lesssim & \eps^3\sum_{k=1}^N\|v_i^n\|_{W^n(t)}^2+\eps^6e^{-k_1(\sigma\eps^3t+L)}.    
  \end{split}
\end{equation}
By \eqref{eq:Aij},
$$\widetilde{\mathcal{A}}_k^n(\mathcal{A}_k^n)^{-1}=
\begin{pmatrix}E_{k,k}\\ D_k\end{pmatrix}+O(e^{-k_1\eps d(t)}),$$
where $E_{k,k}$ is the $2k\times 2k$ identity matrix and 
$D_k$ is a $2(N-k)\times 2k$ matrix whose  $(2j-1)$-th row $(1\le j\le N)$
is a $0$ vector.
Thus by \eqref{eq:modeq1},  \eqref{eq:rkn} and the fact that
$\delta\wR_{2k}^n=O(\eps^{\frac12}e^{-k_1\eps d(t)}\|v_k^n\|_{W_k^n(t)}),$
\begin{equation}
  \label{eq:ham28}
\theta_1(c_i^n)\dot{c}_i^n= c_i^n\la R_{i3}^n,\pd_xu_{c_i^n}\ra
+O\left(\eps^3\sum_{k=1}^N\|v_k^n\|_{X_k^n(t)^*\cap W^n(t)}^2
+\eps^6e^{-k_1\eps d(t)}\right).
\end{equation}
Combining \eqref{eq:ham21}--\eqref{eq:ham28}, we have
\begin{equation}
  \label{eq:v2kenloss}
\left|\frac{d}{dt}H(U_k^n+w_k^n)\right|
\lesssim  \eps^3\sum_{i=1}^N\|v_i^n\|_{X_i^n(t)^*\cap W^n(t)}^2+\eps^6e^{-k_1\eps d(t)}.
\end{equation}
Integrating \eqref{eq:v2kenloss} over $[t,n]$, we obtain
\begin{equation*}
  \begin{split}
& \left|H(U_k^n(t)+w_k^n(t))-H(U_k(n))\right|
 \\ \lesssim & \eps^3\left(\sum_{i=1}^N \int_t^n \|v_i^n(s)\|_{X_i^n(s)^*\cap W^n(s)}^2ds
+e^{-k_1\eps d(t)}\right).
  \end{split}
\end{equation*}
Thus we have \eqref{eq:w2knorm} in the same way as \eqref{eq:enorm}.
This completes the proof of Lemma \ref{lem:speed-Hamiltonian}.
\end{proof}

Next, we will prove virial identities of $w_k^n$.
\begin{lemma}
  \label{lem:virialv2k}
Let $c_{i,+}$ be as in Theorem \ref{thm:1} and let $v_k^n(t)$, $c_k^n(t)$,
$x_k^n(t)$ $(1\le k\le N)$ and $T$ be as in Lemma \ref{lem:parametersupb}.
Let $a=\delta\eps$ and $\pai(t,x)=1-\tanh a(x-x_i^n(t))$.
Then there exists a positive number $C$ such that
for every $T\le t_1\le n$ and $1\le k\le N$,
\begin{equation}
  \label{eq:locdwform}
  \begin{split} &
\la\psi_{a,N}w_k^n(t_1),w_k^n(t_1)\ra_{l^2}+ \eps^{\frac32}\|w_k^n(t)\|_{L^2(t_1,n;W^n(t))}
\\ \le & C\eps^{\frac32}\left(\|v_k^n\|_{L^2(t_1,n;X_k^n(t)^*)}
+\sum_{i=1}^{k-1}\|w_i^n\|_{L^2(0,T;W_k^n(t))}+e^{-k_1\eps d(t)}\right).
  \end{split}
\end{equation}
\end{lemma}
\begin{proof}
Let $u^n={}^t(r^n,p^n)$, $h(u^n)=\frac12(p^n)^2+V(r^n)$ and
$h'(u^n)={}^t(V'(r^n),p^n)$,
$$
H_{k,i}=\la h(U_k^n+w_k^n)-h(U_k^n)-h'(U_k^n)\cdot w_k^n, \pai\ra_{l^2},$$
where $\cdot$ denotes the inner product in $\R^2$. Then
\begin{align*}
  \frac{dH_{k,i}}{dt}=& 
-\dot{x}_i^n\la h(U_k^n+w_k^n)-h(U_k^n)-h'(U_k^n)\cdot w_k^n,\pai'\ra_{l^2}
\\ & +\la H'(U_k^n+w_k^n)-H'(U_k^n),\pai\pd_t(U_k^n+w_k^n)\ra
-\la H''(U_k^n)\pd_tU_k^n,\pai w_k^n\ra
\\ =:I+II.
\end{align*}
By the mean value theorem, there exists a $\theta=\theta(t,n)\in(0,1)$
such that 
$$I= -\frac{\dot{x}_i^n}{2}\la H''(U_k^n+\theta w_k^n)w_k^n,\pai'w_k^n\ra.$$ Since
$\|U_k^n(w_k^n)^2\|_{l^1}\lesssim \eps^2(\|v_k^n\|_{X_k^n(t)^*}+\|w_{k-1}\|_{W_k^n(t)})^2$,
we have
\begin{align*}
I =\frac{\dot{x}_i^n}{2}(1+O(\|w_k^n\|_{l^\infty}))\|\tpai w_k^n\|_{l^2}^2
+O(\eps^3(\|v_k^n\|_{X_k^n(t)^*}+\|w_{k-1}\|_{W_k^n(t)})^2),
\end{align*}
where $\tpai=a^{\frac12}\sech a(x-x_i^n(t))$.
By \eqref{eq:ukwk} and the definition of $U_k^n(t)$,
we have
\begin{align*}
II=& \left\la H'(U_k^n+w_k^n)-H'(U_k^n), \pai JH'(U_k^n+w_k^n)
+\sum_{i=1}^k \pai(l_i^n-P_i^nJR_i^n)\right\ra
\\ & +\la \mathcal{N}_k^n, \pai \tilde{l}_k^n\ra
-\sum_{i=1}^k \la \pai H''(U_k^n)w_k^n,JH'(u_{c_i^n})\ra
\\= \sum_{i=1}^6II_i,
\end{align*}
where  $\mathcal{N}_k^n=H'(U_k^n+w_k^n)-H'(U_k^n)-H''(U_k^n)w_k^n=O((w_k^n)^2)$ and
\begin{align*}
& II_1=\la H'(U_k^n+w_k^n)-H'(U_k^n),\pai J(H'(U_k^n+w_k^n)-H'(U_k^n))\ra,\\
& II_2=\la \mathcal{N}_k^n,\pai JH'(U_k^n)\ra,\quad
 II_3=\la \mathcal{N}_k^n,\pai\tilde{l}_k^n \ra,\\
& II_4= \sum_{i=1}^k\la H'(U_k^n+w_k^n)-H'(U_k^n),\pai l_i^n\ra,\\
& II_5=-\sum_{i=1}^k \la H'(U_k^n+w_k^n)-H'(U_k^n),\pai P_i^nJR_i^n\ra,\\
& II_6=\left\la H''(U_k^n)w_k^n, \pai JU_{k,int}^n\right\ra.
\end{align*}
Following the proof of \cite[Lemma 3.3]{Mi2}, we see that
\begin{align*}
|II_1|\le & \frac{1}{2}\|\tpai(H'(U_k^n+w_k^n)-H'(U_k^n))\|_{l^2}^2
(1+O(a^2))
\\ \le & \|\tpai w_k^n\|_{l^2}^2(1+O(\|w_k^n\|_{l^\infty})
+O(\eps^3(\|v_k^n\|_{X_k^n(t)^*}^2+\|w_{k-1}\|_{W_k^n(t)}^2)),
\end{align*}
whence for a $\delta'>0$,
\begin{align*}
  I+II_1\ge & \frac{\dot{x}_i^n-1+O(\delta\eps^2)}{2}\|\tpai w_k^n\|_{l^2}^2
+O(\eps^3(\|v_k^n\|_{X_k^n(t)^*}+\|w_{k-1}\|_{W_k^n(t)})^2)
\\ \ge &  \delta'\eps^2\|\tpai w_k^n\|_{l^2}^2
+O(\eps^3(\|v_k^n\|_{X_k^n(t)^*}+\|w_{k-1}\|_{W_k^n(t)})^2).
\end{align*}
By Claims \ref{cl:ucsize}--\ref{cl:4} in Appendix \ref{sec:size}
and Lemma \ref{lem:parametersupb},
$$
|II_2|+|II_3|+|II_4| \lesssim
\eps^3\left(\|v_k^n\|_{X_k^n(t)^*}+\sum_{i=1}^{k-1}\|v_i^n\|_{W_k^n(t)}\right)^2.
$$
In view of the formula (2.19) and Claim \ref{cl:ucsize}, we have
$$\sup_n\sup_{t\in[T,n]}\|P_k^nJ\|_{B(l^2,W_k^n(t)^*)}=O(\eps), \quad
\sup_n\sup_{t\in[T,n]}\|P_k^nJ\|_{B(\widetilde{W}_k(t),W_k^n(t)^*)}=O(\eps^{\frac32}),$$
and
\begin{align*}
  |II_5|+|II_6| \lesssim &  \eps^3\left(\|v_k^n\|_{X_k^n(t)^*}
+\sum_{i=1}^{k-1}\|v_i^n\|_{W_k^n(t)}\right)^2+\eps^6e^{-2k_1\eps d(t)}.
\end{align*}
Combining the above, we obtain
\begin{equation}
  \label{eq:v2kvirial}
  \begin{split}
& \delta'\eps^2\|\tpai w_k^n\|_{l^2}^2-\frac{dH_{k,i}}{dt}
\\ \lesssim & \eps^3\left(\|v_k^n\|_{X_k^n(t)^*}+\sum_{i=1}^{k-1}\|w_i^n\|_{W_k^n(t)}
\right)^2+\eps^6e^{-2k_1\eps d(t)}.    
  \end{split}
\end{equation}
Integrating \eqref{eq:v2kvirial} over $[t_1,n]$ and summing up 
the resulting equations for $1\le i\le N$,
we have \eqref{eq:locdwform} since $H_{k,i}=\|\pai^{\frac12}w_k^n\|_{l^2}^2
(1+O(\|U_k^n\|_{l^\infty}+\|w_k^n\|_{l^\infty}))$.
Thus we prove Lemma \ref{lem:virialv2k}.
\end{proof}

We use the following exponential stability property of the linearized FPU
equation.
\begin{lemma}
  \label{lem:linearstability}
Let $\zeta={}^t(\zeta_1,\zeta_2)\in C^1(\R^2)$,
 $\mathcal{F}_n\zeta\in L^1(\T)$,
$F_1$, $F_2\in C(\R;l^2_{-k_1\eps})$ and let
$w(t)\in C^1(\R;l^2_{k_1\eps})$ be a solution of 
\begin{equation}
  \label{eq:LFPU2}
\pd_tw(t)=JH''(\widetilde{U}_k(t)+\zeta(t))w(t)+F_1(t)+JF_2(t).
\end{equation}
There exist positive numbers $\eps_0$, $L_0$, $\delta_1$, $\delta_2$, $M$ and
$b$ satisfying the following: Suppose $\eps\in(0,\eps_0)$,
$0\le T_1\le T_2\le \infty$ and that
\begin{gather*}
\inf_{t\in[T_1,T_2]}\min_{2\le j\le N}\eps(x_{j,+}(t)-x_{j-1,+}(t))\ge L_0,\\
\sup_{t\in[T_1,T_2]}\sup_{x\in\R}
(|\zeta_1(t,x)|+\eps^{-1}|\pd_x\zeta_1(t,x)|)\le \delta_1\eps^2,
\end{gather*}
and
\begin{equation}
  \label{eq:orth3'}
  \begin{split}
 & \eps^{-\frac32}|\la w(t), J^{-1}\pd_xu_{c_{i,+}}(\cdot-x_{i,+}(t))\ra|
+\eps^{\frac32}|\la w(t),J^{-1}\pd_cu_{c_{i,+}}(\cdot-x_{i,+}(t))\ra|
\\ \le &  \delta_2\|w(t)\|_{X_k^n(t)^*}
\end{split}
\end{equation}
for $1\le i\le k$ and  $t\in[T_1,T_2]$.
Then for every $t, t_1\in[T_1,T_2]$ satisfying $t\le t_1$,
\begin{align*}
 \|w(t)\|_{X_k^n(t)^*}  \le &
Me^{b\eps^3(t-t_1)}\|w(t_1)\|_{X_k^n(t_1)}
\\ &+M\int^{t_1}_t e^{b\eps^3(t-s)}
(\|F_1(s)\|_{X_k^n(s)^*}+\eps^{-\frac12} (s-t)^{-\frac12}\|F_2(s)\|_{X_k^n(s)^*})ds.
\end{align*}
\end{lemma}
\begin{proof}
Let $0<k_1<\cdots<k_N$, $\gamma_i\in\R$, 
$\theta_i=k_i(x-4k_i^2t-\gamma_i)$ for $i=1,\dots,N$ and let 
$\mk=(k_1,\cdots,k_N)$, $\mgamma=(\gamma_1,\ldots, \gamma_N)\in\R^N$ and
$$C_N=\left[\frac{1}{k_i+k_j}e^{-(\theta_i+\theta_j)}
\right]_{\substack{i=1,\dots,N\downarrow\\ j=1,\cdots,N\rightarrow}}.$$
Then $\varphi_N(t,x;\mk,\mgamma):=\pd_x^2\log\det(I+C_N)$ is
an N-soliton solution of KdV
\begin{equation}
  \label{eq:KdV}
  \pd_tu+\pd_x(\pd_x^2u+6u^2)=0.
\end{equation}
\par
Let $0<a<2k_1$ and
\begin{gather*}
\mathcal{P}(t,\mk,\mgamma):L^2_{-a}\to
\spann\{\pd_{\gamma_i}\varphi_{N}(t,y;\mk,\mgamma),\;
\pd_{k_i}\varphi_{N}(t,y;\mk,\mgamma)\,:\, 1\le i\le N\},
\\  \mathcal{Q}(t,\mk,\mgamma)=I-\mathcal{P}(t,\mk,\mgamma)
\end{gather*}
be projections associated with
\begin{equation}
  \label{eq:LKdV}
\pd_tv+\pd_x(\pd_x^2v+12\varphi_N(\mk,\mgamma)v)=0
\end{equation}
such that for $v\in\mathcal{Q}(t,\mk,\mgamma)$ and $i=1,\cdots,N$,
\begin{align*}
& \int_\R v(x)\int_x^\infty \pd_{\gamma_i}\varphi_N(t,y;\mk,\mgamma)dydx=0,\\
& \int_\R v(x)\int_x^\infty\pd_{k_i}\varphi_N(t,y;\mk,\mgamma)dydx=0.
\end{align*}
It follows from \cite[Section 6]{Mi2} that for every $t\le s$ and $c$,
$x_0\in\R$,
\begin{equation}
  \label{eq:lkdvest}
\|e^{-a(\cdot-ct-x_0)}\mathcal{Q}(t)v(t)\|_{L^2}\lesssim
e^{a(c-a^2)(t-s)}\|e^{-a(\cdot-cs-x_0)}\mathcal{Q}(s)v(s)\|_{L^2}  
\end{equation}
if $v(t)$ is a solution of \eqref{eq:LKdV}.
Applying \eqref{eq:lkdvest} to the low frequency part of a solution to
\eqref{eq:LFPU2} and applying a semigroup estimate
$$\|e^{-a(n-ct)}e^{tJ}f\|_{l^2}\lesssim e^{(ca-2\sinh(a/2))t}\|e^{-an}f\|_{l^2}
\quad\text{for $t\le0$,}$$
to the high frequency part of the solution, we obtain
Lemma \ref{lem:linearstability} in exactly the same way as the proof of
\cite[Lemma 5.1]{Mi2}.
\end{proof}
Now we will estimate remainder parts of solutions $u_n$ in exponentially
weighted space. 
\begin{lemma}
\label{lem:vNXN}
 Let $c_{i,+}$ be as in Theorem \ref{thm:1} and let $v_k^n(t)$, $c_k^n(t)$,
$x_k^n(t)$ $(1\le k\le N)$ and $T$ be as in Lemma \ref{lem:parametersupb}.
Then there exists positive constants $A_i$ $(1\le i\le N)$ such that
  \begin{equation}
    \label{eq:vkXk}
\|v_i^n(t)\|_{X_k^n(t)^*}\le A_i\eps^{\frac32}e^{-k_1\eps d(t)}
\quad\text{for every $t\in [T,n]$ and $1\le i\le N$.}
  \end{equation}
\end{lemma}
\begin{proof}
To begin with, we will estimate difference between $x_i^n(t)$ and $x_{i,+}(t)$
assuming \eqref{eq:vkXk}.
Let $T_1(n):=\inf\{\tau\ge T:\text{\eqref{eq:vkXk} holds for
$\tau \le t \le n$}\}.$ 
Lemma \ref{lem:virialv2k} and \eqref{eq:xcupb} imply that for $t\in[T_1(n),n]$,
\begin{align*}
  \eps^{-^3}|\dot{c}_i^n(t)|+|\dot{x}_i^n(t)-c_i^n(t)|
\lesssim &
\eps^{\frac12}(\|v^n(t)\|_{W^n(t)}+\eps^{\frac32}e^{-k_1\eps d(t)})
\\ \lesssim & \eps^2\left(\sum_{i=1}^k \|v_i^n\|_{L^2(t,n;X_k^n(t)^*)}
+ e^{-k_1\eps d(t)}+e^{-k_1\eps d(n)}\right)
\\ \lesssim & \eps^2e^{-k_1\eps d(t)},
\end{align*}
whence for $t\in[T_1(n),n]$,
\begin{equation}
  \label{eq:diffx+xi}
  \begin{split}
|x_i^n(t)-x_{i,+}(t)|\le & \int^n_t
\left(|\dot{x}_i^n(s)-c_i^n(s)|+
\int_s^n|\dot{c}_i^n(\tau)|d\tau\right)ds
\\ \lesssim & \eps^2\int_t^n e^{-k_1\eps d(s)}ds \lesssim \eps^{-1}e^{-k_1L},
\end{split}
\end{equation}
\begin{equation}
  \label{eq:diffc+ci}
  \begin{split}
|c_i^n(t)-c_{i,+}(t)|\le & \int^n_t|\dot{c}_i^n(s)|ds
\\ \lesssim & \eps^2e^{-k_1L}.
\end{split}
\end{equation}
\par

Now we will estimate $\|v_k^n(t)\|_{X_k^n(t)^*}$ by induction. Suppose that
\eqref{eq:vkXk} holds for $i<k$ and $t\in[T_1(n),n]$.
Let $\zeta(t)=\sum_{i=1}^k\{u_{c_i^n(t)}^n(\cdot-x_i^n(t))
-u_{c_{i,+}}^n(\cdot-x_{i,+}(t))\}$.
By Claim \ref{cl:ucsize}, \eqref{eq:diffx+xi} and \eqref{eq:diffc+ci}, we have
$|\pd_x^j\zeta(t,x)|\lesssim e^{-k_1L}\eps^{j+2}$ for $j=0$, $1$.
Applying Lemma \ref{lem:linearstability}
to \eqref{eq:vk}, we see that for $t\in[T_1(n),n]$ and $1\le k\le N$,
\begin{equation}
  \label{eq:v2kint1}
  \begin{split}
\|v_k^n(t)\|_{X_k^n(t)^*}\lesssim  & \int_t^n e^{b\eps^3(t-s)}
\left(\|l_k^n(s)\|_{X_k^n(s)^*}+\|[Q^n(s),J]R_k^n\|_{X_k^n(s)^*}\right)ds
\\ & +\eps^{-\frac12}\int_{t_j}^te^{b\eps^3(t-s)}(s-t)^{-\frac12}
\|Q^n(s)R_k^n\|_{X_k^n(s)^*}ds.  
\end{split}
\end{equation}
By the definition of $l_k^n$ and \eqref{eq:alphabetaupb},
\begin{equation}
  \label{eq:lkn}
  \begin{split}
\|l_k^n\|_{X_k^n(t)^*}\lesssim & 
\eps^{\frac32}\|v_k^n\|_{X_k^n(t)^*}\left(\sum_{i=1}^{N}\|v_i^n\|_{W^n(t)}
+\eps^{\frac32}e^{-k_1\eps d(t)}\right)
\\ \lesssim & (\delta+e^{-k_1L})\eps^3\|v_k^n\|_{X_k^n(t)^*}.      
  \end{split}
\end{equation}
By Lemma \ref{lem:virialv2k} and the induction hypothesis,
$$\|w_{k-1}^n\|_{W^n(t)} \lesssim
\eps^{\frac32}\left(\sum_{i=1}^{k-1}\|v_i^n\|_{L^2(t,n;X_i^n(t)^*)}+e^{-k_1\eps d(t)}\right)
\lesssim \eps^{\frac32}e^{-k_1\eps d(t)}.$$
Hence it follows from \eqref{eq:rkn} that
 \begin{align*}
\|R_k^n\|_{X_k^n(t)^*}\lesssim & 
\|R_{k1}^n\|_{X_k^n(t)^*}+\|R_{k2}^n\|_{X_k^n(t)^*}+\|R_{k3}^n\|_{X_k^n(t)^*}
\\ \lesssim & 
\|v_k^n\|_{X_k^n(t)^*}(\|v_k^n\|_{l^2}+ \|w_{k-1}^n\|_{l^2})
+\eps^{\frac72}e^{-k_1\eps d(t)}+\eps^2\|w_{k-1}^n\|_{W^n(t)}
\\ \lesssim &  \eps^2\delta\|v_k^n\|_{X_k^n(t)^*}+\eps^{\frac72}e^{-k_1\eps d(t)}.
 \end{align*}
Substituting the above inequalities and $\|[Q^n(s),J]\|_{B(X_k^n(s)^*)}=O(\eps)$
into \eqref{eq:v2kint1}, we have
\begin{equation}
  \label{eq:<n}
  \begin{split}
&  \|v_k^n(t)\|_{X_k^n(t)^*} \\ \lesssim &
\eps^{\frac92}\int_t^n e^{b\eps^3(t-s)}
(1+\eps^{-\frac32}(s-t)^{-\frac12})e^{-k_1\eps d(s)}ds
\\ & + (\delta+e^{-k_1L})\eps^3\int_t^ne^{b\eps^3(t-s)}
(1+\eps^{-\frac32}(s-t)^{-\frac12})\|v_k^n(s)\|_{X_k^n(s)^*}ds
\\ \le &
C_1\eps^{\frac32}e^{-k_1\eps d(t)}+C_2(\delta+e^{-k_1L})
\sup_{s\in[t,n]}\|v_k^n(s)\|_{X_k^n(s)^*},
  \end{split}
\end{equation}
where $C_1$ and  $C_2$ are positive constants independent of $n$ and $t
\in[T_1(n),n]$. Thus we have
\begin{equation}
  \label{eq:v2kint2}
\|v_k^n(t)\|_{X_k^n(t)^*} \le 2C_1\eps^{\frac32}e^{-k_1\eps d(t)}
\end{equation}
for $t\in [T_1(n),n]$ and $1\le k\le N$ provided $C_2(\delta+e^{-k_1L})\le 1/2$.
Letting  $A_k=2C_1$, we have $T_1=T$  from  \eqref{eq:v2kint2}.
Thus we complete the proof of Lemma \ref{lem:vNXN}.
\end{proof}

Let $\gamma_i^n(t):=x_i^n(t)-c_{i,+}t$.
A system of \eqref{eq:xcinit} and \eqref{eq:modeq1} can be rewritten
as
\begin{equation}
  \label{eq:modeq2}
  \begin{split}
& \begin{pmatrix}c_i^n(t) \\ \gamma_i^n(t)\end{pmatrix}_{1\le i\le N\downarrow}
-\begin{pmatrix}c_{i,+} \\ \gamma_{i,+}\end{pmatrix}_{1\le i\le N\downarrow}
\\= &  \mathcal{E}_1 \int_t^n
\left(\mathcal{A}_N^n-\sum_{k=1}^N\delta\widetilde{\mathcal{A}}_k^n\right)^{-1}
\sum_{k=1}^N\left(\delta\wR_{1k}^n+\delta\wR_{2k}^n\right)ds
\\ & -\mathcal{E}_2 \int_t^n\int_s^n
\left(\mathcal{A}_N^n-\sum_{k=1}^N\delta\widetilde{\mathcal{A}}_k^n\right)^{-1}
\sum_{k=1}^N\left(\delta\wR_{1k}^n+\delta\wR_{2k}^n\right)d\tau ds,
  \end{split}
\end{equation}
where 
$\mathcal{E}_1$ and $\mathcal{E}_2$ are $2N\times 2N$ matrices such that
$$\mathcal{E}_1=\diag(\mathbf{p_1},\dots,\mathbf{p_1}),
\enskip \mathcal{E}_2=\eps^3\diag(\mathbf{p_2},\dots,\mathbf{p_2}),\enskip
\mathbf{p_1}=\begin{pmatrix} \eps^3 & 0 \\ 0 & -1\end{pmatrix},\enskip
\mathbf{p_2}=\begin{pmatrix} 0 & 0 \\ 1 & 0\end{pmatrix}.$$
Combining Lemmas \ref{lem:parametersupb}, \ref{lem:speed-Hamiltonian},
\ref{lem:virialv2k} and \ref{lem:vNXN}, we obtain the following.
\begin{proposition}
  \label{prop:vNupb}
 Let $c_{i,+}$ be as in Theorem \ref{thm:1}. Then there exist positive constants
$\varepsilon_0$, $\delta$, $T$ and $A$
such that for every $k=1, \cdots, N$, $t\in[T,n]$ and $n>T$,
\begin{gather*}
\|v_k^n\|_{X_k^n(t)^*}+\|v_k^n\|_{W^n(t)}+\|v_k^n\|_{l^2}\le A\eps^{\frac32}e^{-k_1\eps d(t)},
\\
|c_k^n(t)-c_{k,+}|+\eps^3|\gamma_k^n(t)-\gamma_{k,+}|\le A\eps^2e^{-k_1\eps d(t)}.
\end{gather*}
\end{proposition}
\begin{proof}
Suppose $n_0\in\N$ is sufficiently large. Then for $n\ge n_0$,
there exists a $T_n<n$ such that
\begin{gather*} \sup_{t\in[T_n,n]}\left\{|c_i^n(t)-c_{i,+}|
+\sum_{k=1}^N(\|v_k^n(t)\|_{X_k^n(t)^*}+\|v_k^n(t)\|_{l^2})
\right\}\le \delta\eps^2,
\\ \inf_{t\in[T_n,n]}\min_{2\le i\le N}\eps(x_{i}^n(t)-x_{i-1}^n(t))\ge L_0.
\end{gather*}
Lemmas \ref{lem:parametersupb}, \ref{lem:speed-Hamiltonian},
\ref{lem:virialv2k} and \ref{lem:vNXN}  imply that for $t\in[T_n,n]$
and $n\ge n_0$,
\begin{gather}
\label{eq:vbn}
\sum_{k=1}^N(\|v_k^n\|_{X_k^n(t)^*}+\|v_k^n\|_{W^n(t)}+\|v_k^n\|_{l^2})\le
C\eps^{\frac32}e^{-k_1\eps d(t)},\\
\label{eq:cgbn}
\sum_{k=1}^N(|\dot{c}_k^n(t)|+\eps^3|\dot{\gamma}_k^n(t)|)\le
C\eps^2e^{-k_1\eps d(t)},
\end{gather}
where $C$ is a positive constant independent of $n\ge n_0$.
By \eqref{eq:cgbn}, there exists a positive constant $C'$ such that
\begin{equation}
  \label{eq:cgbn2}
\sup_{t\in[T_n,n]}(|c_k^n(t)-c_{k,+}|+\eps^3|\gamma_k^n(t)-\gamma_{k,+}|)
\le C'\eps^2e^{-k_1\eps d(T_n)}\quad\text{for every $n\ge n_0$.}
\end{equation}
In view of \eqref{eq:vbn} and \eqref{eq:cgbn2}, we see that $T_n$ can
be chosen independently of $n\ge n_0$. This completes the proof of
Proposition \ref{prop:vNupb}.
\end{proof}

\section{Existence of $N$-soliton like solutions}
\label{sec:existence}
In this section, we will show that as $n\to\infty$, a solution $u^n(t)$ of
\eqref{eq:un} converges to an $N$-soliton state, that is,
a solution of \eqref{eq:FPU} that tends to a sum of $N$-solitary waves as
$t\to\infty$.
\par
To begin with, we introduce several notations. Let
\begin{gather*}
\|u\|_{Y_n(t)}:=\sup_{s\in[t,n]}e^{\frac{k_1\eps d(s)}2}\|u(s)\|_{l^2},
\quad 
\|u\|_{Z_{k,n}(t)}:=\sup_{s\in[t,n]}e^{\frac{k_1\eps d(s)}2}\|u(s)\|_{X_k^n(s)^*},\\
\dc_k^{m,n}(t)=c_k^m(t)-c_k^n(t),\quad \dg_k^{m,n}(t)=\gamma_k^m(t)-\gamma_k^n(t),\\
\dv_k^{m,n}(t)=v_k^m(t)-v_k^n(t),\quad \dw_k^{m,n}(t)=w_k^m(t)-w_k^n(t),\\
\du_k^{m,n}(t)=u_{c_k^m(t)}(\cdot-x_k^m(t))-u_{c_k^n(t)}(\cdot-x_k^n(t)),
\quad \dU_k^{m,n}(t)=U_k^m(t)-U_k^n(t).
\end{gather*}
First, we will prove that $\{c_k^n\}_{n=1}^\infty$ and $\{\gamma_k^n\}_{n=1}^\infty$
are Cauchy sequences assuming that $\{v_k^n\}$ is a Cauchy sequence in 
a weighted space.
\begin{lemma}
  \label{lem:cauchyxc}
There exist an $n_0\in\N$ and positive constants $C$ and $T$ such that for any
$m\ge n\ge n_0$ and  $t\in [T,n]$,
\begin{align*}
& \sup_{s\in[t,n]}e^{\frac12k_1\eps d(s)}(|\dc^{m,n}_k(s)|+\eps^3|\dg^{m,n}_k(s)|)
\\ \le  & C\eps^2e^{-\frac12k_1\eps d(n)}+
C\eps^{\frac12}\left(\sum_{i=1}^{k-1}\|\dv_i^{m,n}\|_{Y_n(t)}
+e^{-k_1\eps d(t)}\sum_{i=1}^N(\|\dv_i^{m,n}\|_{Y_n(t)}+\|\dv_i^{m,n}\|_{Z_{i,n}(t)})\right).
\end{align*}
\end{lemma}
\begin{proof}
By the definition, $|\delta R_{1k}^m-\delta R_{1k}^n|\lesssim I_{k}+II_{k}$,
where
\begin{align*}
I_{k}= & \|R_{k}^n\|_{l^1}\sum_{i=1}^k(\eps^{-4} \|\pd_x\du_i^{m,n}\|_{l^\infty} 
+\eps^{-1}\|\pd_cu_{c_i^m}(\cdot-x_i^m)-\pd_cu_{c_i^n}(\cdot-x_i^n)\|_{l^\infty}),\\
II_{k}=& \sum_{j=1}^3\|R_{kj}^m-R_{kj}^n\|_{l^1}\sum_{i=1}^k
\left(\eps^{-4}\|\pd_xu_{c_i^n}(\cdot-x_i^n)\|_{l^\infty}
+\eps^{-1}\|\pd_cu_{c_i^n}(\cdot-x_i^n)\|_{l^\infty}\right).
\end{align*}
By Claim \ref{cl:ucsize}, we have
\begin{equation}
  \label{eq:ucdiff}
|\pd_x^\alpha\pd_c^\beta u_{c_i^m}(\cdot-x_i^m)
-\pd_x^\alpha\pd_c^\beta u_{c_i^n}(\cdot-x_i^n)|
\lesssim \eps^{\alpha-2\beta}(|\dc_i^{m,n}|+\eps^3|\dg^{m,n}_i|)
e^{-\frac32k_i|\cdot-x_i^n|}.
\end{equation}
In view of Proposition \ref{prop:vNupb}, \eqref{eq:rkn} and
\eqref{eq:ucdiff}, we have
\begin{equation}
\label{eq:lemxc1}
I_k\lesssim e^{-k_1\eps d(t)}
\sum_{i=1}^N\left(|\dc^{m,n}_i(t)|+\eps^3|\dg^{m,n}_i(t)|\right).
\end{equation}
\par
\par
Next we will estimate $II_{k}$.
By the mean value theorem, we have
$$R_{k1}^n=\int_0^1\int_0^1H'''(U_k^n+\theta_2(w_{k-1}^n+\theta_1v_k^n))
v_k^n(w_{k-1}^n+\theta_1v_k^n) d\theta_1d\theta_2,$$
and it follows that
\begin{equation}
  \label{eq:rk1dif}
  \begin{split}
|R_{k1}^m-R_{k1}^n| \lesssim
 & |v_k^n|(|\dw^{m,n}_{k-1}|+|\dv_k^{m,n}|) +|\dv_k^{m,n}|(|w_{k-1}^m|+|v_k^m|)
\\ & +|v_k^n|(|w_{k-1}^n|+|v_k^n|)|\dU_k^{m,n}|.    
  \end{split}
\end{equation}
Similarly, we have
\begin{align}
\label{eq:rk2dif}
|R_{k2}^m-R_{k2}^n|\lesssim &
|\dU_{k-1}^{m,n}||u_{c_k}^m|+|U_{k-1}^n||\du_k^{m,n}|,
\\ \label{eq:rk3dif} 
|R_{k3}^m-R_{k3}^n|\lesssim & |u_{c_k^m}||\dw^{m,n}_{k-1}|+|w_{k-1}^n ||\du_k^{m,n}|
+|w_{k-1}^m||u_{c_k}^m||\dU_{k-1}^{m,n}|.
\end{align}
Substituting \eqref{eq:ucdiff} into the above and using Proposition
\ref{prop:vNupb}, we obtain
\begin{align*}
& \|R_{k1}^m-R_{k1}^n\|_{l^1}\\ \lesssim  & 
\sum_{i=1}^k(\|v_i^m\|_{l^2}+\|v_i^n\|_{l^2})\sum_{i=1}^k\|\dv_i^{m,n}\|_{l^2}
+\sum_{i=1}^k\|v_i^n\|_{l^2}^2\sum_{i=1}^k(|\dc^{m,n}_i|+\eps^3|\dg^{m,n}_i|)
\\ \lesssim &
e^{-k_1\eps d(t)}\sum_{i=1}^k\left\{\eps^{\frac32}\|\dv_i^{m,n}\|_{l^2}
+\eps^3(|\dc^{m,n}_i|+\eps^3|\dg^{m,n}_i|)\right\},
\end{align*}
\begin{align*}
\|R_{k2}^m-R_{k2}^n\|_{l^1}\lesssim & 
\eps e^{-k_1\eps d(t)}\sum_{i=1}^k(|\dc^{m,n}_i|+\eps^3|\dg^{m,n}_i|),\\
\|R_{k3}^m-R_{k3}^n\|_{l^1}\lesssim &
\eps^{\frac32}\sum_{i=1}^{k-1}\|\dv_i^{m,n}\|_{l^2}
+\eps e^{-k_1\eps d(t)}\sum_{i=1}^{k}(|\dc^{m,n}_i|+\eps^3|\dg^{m,n}_i|).
\end{align*}
Combining the above, we obtain
\begin{align}
  \label{eq:lemxc2}
& II_{k} \lesssim \eps^{\frac12}\sum_{i=1}^{k-1}\|\dv_i^{m,n}\|_{l^2}
+e^{-k_1\eps d(t)}\sum_{i=1}^k(|\dc^{m,n}_i|+\eps^3|\dg^{m,n}_i|).
\end{align}
It follows from Proposition \ref{prop:vNupb} and \eqref{eq:ucdiff} that
\begin{equation}
  \label{eq:lemxc4}
  \begin{split}
& |\delta R_{2k}^m-\delta R_{2k}^n| \\  \lesssim &
\eps^{\frac12}e^{-k_1\eps d(t)}\|\dv_k^{m,n}\|_{l^2}
+\|v_k^n\|_{l^2}\sum_{j=1}^k(\|u_{c_j^m}\|_{l^\infty}+\|u_{c_j^n}\|_{l^\infty})
\\ & \quad \times
\sum_{i=1}^k(\eps^{-4}\|\pd_x\du_i^{m,n}\|_{l^2}+\eps^{-1}\|\pd_cu_{c_i^m}(\cdot-x_i^m)
-\pd_cu_{c_i^n}(\cdot-x_i)\|_{l^2})
\\  \lesssim & e^{-k_1\eps d(t)}
\left(\eps^{\frac12}\|\dv_k^{m,n}\|_{l^2}+\sum_{i=1}^k(|\dc^{m,n}_i|+\eps^3|\dg^{m,n}_i|)\right),
  \end{split}
\end{equation}
and 
\begin{equation}
  \label{eq:lemxc5}
  \begin{split}
& \left|\mathcal{A}_N^m-\mathcal{A}_N^n\right| +\sum_{k=1}^N
\left|\delta\widetilde{\mathcal{A}}_k^m-\delta\widetilde{\mathcal{A}}_k^n\right|
\\  \lesssim & \sum_{k=1}^N \left(\eps|\dg^{m,n}_k|
+\eps^{-2}|\dc^{m,n}_k|+\eps^{-\frac32}\|\dv_k^{m,n}\|_{X_k^n(t)^*}\right).
 \end{split}
\end{equation}
Combining \eqref{eq:modeq2}, 
\eqref{eq:lemxc2}, \eqref{eq:lemxc2}--\eqref{eq:lemxc5}, we have
\begin{equation}
  \label{eq:dotx-cdif}
  \begin{split}
& \left|\pd_t\dg^{m,n}_k(t)-\dc^{m,n}_k(t)\right|+\eps^{-3}\left|\pd_t\dc^{m,n}_k(t)\right|
 \lesssim  e^{-k_1\eps d(t)}\sum_{i=1}^N(|\dc^{m,n}_i|+\eps^3|\dg^{m,n}_i|)
\\ & \quad +\eps^{\frac12}\left\{\sum_{i=1}^{k-1}\|\dv_i^{m,n}\|_{l^2}
+e^{-k_1\eps d(t)}\sum_{i=1}^N\|\dv_i^{m,n}\|_{X_i^n(t)^*\cap l^2}\right\}.
\end{split}
\end{equation}
Integrating the above on $[t,n]$ and using
$\eps^3\int_t^ne^{-k_1\eps d(s)}ds \lesssim e^{-k_1\eps d(t)}$, we have
\begin{align*}
& |\dc^{m,n}_k(t)|+\eps^3\left|\dg^{m,n}_k(t)-\int_n^t\dc^{m,n}_k(s)ds\right|
\\ \lesssim & |\dc^{m,n}_k(n)|+\eps^3|\dg^{m,n}_k(n)|+
e^{-k_1\eps d(t)}\sup_{s\in[t,n]} (|\dc^{m,n}_i(s)|+\eps^3|\dg^{m,n}_i(s)|)
\\ & +\eps^{\frac12}e^{-\frac12k_1\eps d(t)}\left(\sum_{i=1}^{k-1}\|\dv_i^{m,n}\|_{Y_n(t)}
+e^{-k_1\eps d(t)}\sum_{i=1}^N\|\dv_i^{m,n}\|_{Y_n(t)\cap Z_{i,n}(t)}\right).
\end{align*}
Now Lemma \ref{lem:cauchyxc} follows immediately from the above since
 $$|\dc^{m,n}_k(n)|+\eps^3|\dg^{m,n}_k(n)|\lesssim \eps^2e^{-k_1\eps d(n)}$$
by Proposition \ref{prop:vNupb}.
\end{proof}

Next we will show that $\{v_k^n\}_{n=1}^\infty$ is a Cauchy sequence in
exponentially weighted spaces if it is a Cauchy sequence in the energy space.
\begin{lemma}
  \label{lem:vkxkdif}
There exist positive constants $C$, $b$, $T$ and an $n_0\in\N$ such that
for any $m\ge n\ge n_0$ and  $t\in [T,n]$,
\begin{align*}
 \|\dv_k^{m,n}\|_{Z_{k,n}(t)} \le & C\eps^{\frac32}e^{-\frac12k_1\eps d(n)}
+C\left(\sum_{i=1}^{k-1}\|\dv_i^{m,n}\|_{Y_n(t)}
+\eps^{-\frac12}e^{-k_1\eps d(t)}\sum_{i=1}^{N}\|\dv_i^{m,n}\|_{Y_n(t)}\right).
\end{align*}
\end{lemma}
By Lemmas \ref{lem:cauchyxc} and \ref{lem:vkxkdif}, we have the following.
\begin{corollary}
  \label{cor:ucdiff}
  \begin{align*}
& \sup_{s\in[t,n]}e^{\frac12k_1\eps d(s)}\|\dU_k^{m,n}(s)\|_{l^2} 
 \lesssim \eps^{-\frac12}\sup_{s\in[t,n]}e^{\frac12k_1\eps d(s)}
(|\dc^{m,n}_k(s)|+\eps^3|\dg^{m,n}_k(s)|)
\\ \lesssim & 
\eps^{\frac32}e^{-\frac12k_1\eps d(n)}+\sum_{i=1}^{k-1}\|\dv_i^{m,n}\|_{Y_n(t)}
+ \eps^{-\frac12}e^{-k_1\eps d(t)}\sum_{i=1}^N \|\dv_i^{m,n}\|_{Y_n(t)}.
  \end{align*}
\end{corollary}
\begin{proof}[Proof of Lemma \ref{lem:vkxkdif}]
In view of \eqref{eq:pkdef} and Lemma \ref{lem:cauchyxc},
\begin{equation}
  \label{eq:difpro}
  \begin{split}
\|P_k^m-P_k^n\|_{B(X_k^n(t)^*)} \lesssim & \sum_{i=1}^k(\eps^{-2}|\dc^{m,n}_i|+\eps|\dg^{m,n}_i|)
\\ \lesssim &  \eps^{-\frac32}e^{-\frac12k_1\eps d(t)} 
\left(\eps^{\frac32}e^{-\frac12k_1\eps d(n)}
+\sum_{i=1}^N \|\dv_i^{m,n}\|_{Y_n(t)\cap Z_{k,n}(t)}\right).
  \end{split}
\end{equation}
Thus by Lemma \ref{lem:vNXN}, we have
\begin{equation}
  \label{eq:difvmvnp}
  \begin{split}
\|P_k^n(t)\dv_k^{m,n}(t)\|_{X_k^n(t)^*}=&\|(P_k^n-P_k^m)v_k^m\|_{X_k^n(t)^*}
\\  \lesssim & e^{-\frac32k_1\eps d(t)}
\left(\eps^{\frac32}e^{-\frac12k_1\eps d(n)}+\sum_{i=1}^N \|\dv_i^{m,n}\|_{Y_n(t)\cap Z_{k,n}(t)}
\right).\end{split}
\end{equation}
\par

Now we will estimate $\|Q_k^n(t)\dv_k^{m,n}(t)\|_{X_k^n(t)^*}$ by using
Lemma \ref{lem:linearstability}.
By \eqref{eq:vk},
\begin{equation}
  \label{eq:qdifv}
  \pd_t(Q_k^n\dv_k^{m,n})=JH''(U_k^n)Q_k^n\dv_k^{m,n}+I+II+III+IV,
\end{equation}
where
\begin{align*}
I=& \dot{Q}_k^n\dv_k^{m,n}-[JH''(U_k^n),Q_k^n]\dv_k^{m,n},\\
II=& Q_k^nJ(H''(U_k^m)-H''(U_k^n))v_k^m,\\
III=& Q_k^n(l_k^m-l_k^n)=(P_k^m-P_k^n)l_k^m,\\
IV=& Q_k^n(Q_k^mJR_k^m-Q_k^nJR_k^n)
\\=& Q_k^nJ(R_k^m-R_k^n)-(P_k^m-P_k^n)P_k^mJR_k^m.
\end{align*}
We decompose as $I$ as $I=I_1+I_2+I_3+I_4$ and
\begin{align*}
I_1=& -(\eps^3\pd_cu_{c_j^n},\pd_xu_{c_j^n})_{j=1,\cdots,k\rightarrow}
\frac{d}{dt}(\mathcal{A}_k^n)^{-1}
\begin{pmatrix}
\eps^{-4}\la \dv_k^{m,n},J^{-1}\pd_xu_{c_i^n}\ra \\
\eps^{-1}\la \dv_k^{m,n},J^{-1}\pd_cu_{c_i^n}\ra
\end{pmatrix}_{i=1,\cdots,k \downarrow},
\\ I_2=&
(\eps^3\pd_cu_{c_j^n},\pd_xu_{c_j^n})_{j=1,\cdots,k\rightarrow}
(\mathcal{A}_k^n)^{-1} 
\begin{pmatrix} 
\eps^{-4}\la \dv_k^{m,n}, \Delta_{12}(i)-\Delta_{11}(i) \ra
 \\ \eps^{-1}\la \dv_k^{m,n}, \Delta_{22}(i)-\Delta_{21}(i) \ra
\end{pmatrix}_{i=1,\cdots,k \downarrow},
\\ I_3=&
J(\eps^3(\Delta_{22}(j)-\Delta_{21}(j)),
\Delta_{12}(j)-\Delta_{11}(j))_{j=1,\cdots,k\rightarrow}\\
& \qquad \times
(\mathcal{A}_k^n)^{-1} 
\begin{pmatrix}
\eps^{-4}\la \dv_k^{m,n},J^{-1}\pd_xu_{c_i^n}\ra \\
\eps^{-1}\la \dv_k^{m,n},J^{-1}\pd_cu_{c_i^n}\ra
\end{pmatrix}_{i=1,\cdots,k \downarrow},
\\ I_4=&
-(\eps^3\pd_cu_{c_j^n},\pd_xu_{c_j^n})_{j=1,\cdots,k\rightarrow}
\left(\mathcal{E}_2(\mathcal{A}_k^n)^{-1}+(\mathcal{A}_k^n)^{-1}\mathcal{E}_2
\right)
\begin{pmatrix}
\eps^{-4}\la \dv_k^{m,n},J^{-1}\pd_xu_{c_i^n}\ra \\
\eps^{-1}\la \dv_k^{m,n},J^{-1}\pd_cu_{c_i^n}\ra
\end{pmatrix}_{i=1,\cdots,k \downarrow},
\end{align*}  
where
\begin{align*}
& \Delta_{11}(i)=
\dot{c}_i^nJ^{-1}\pd_c\pd_xu_{c_i^n}+(c_i^n-\dot{x}_i^n)J^{-1}\pd_x^2u_{c_i^n},
\quad 
\Delta_{12}(i)=(H''(U_k^n)-H''(u_{c_i^n}))\pd_xu_{c_i^n},
\\
& \Delta_{21}(i)=
\dot{c}_i^nJ^{-1}\pd_c^2u_{c_i^n}+(c_i^n-\dot{x}_i^n)J^{-1}\pd_c\pd_xu_{c_i^n},
\quad \Delta_{22}(i)=(H''(U_k^n)-H''(u_{c_i^n}))\pd_cu_{c_i^n}.
\end{align*}
By \eqref{eq:xcupb}, we have
\begin{equation}
  \label{eq:daij}
  \frac{d}{dt}\mathcal{A}_{i,j}^n
=\begin{cases}
& O(\eps^3e^{-k_1\eps|x_i^n-x_j^n|}) \quad\text{if $i\ne j$,}\\
& O(\eps^3e^{-k_1\eps d(t)})\quad\text{if $i=j$.}
\end{cases}
\end{equation}
Indeed,
\begin{align*}
& \frac{d}{dt}\la \pd_cu_{c_i^n}(\cdot-x_i^n),J^{-1}\pd_cu_{c_j^n}(\cdot-x_j^n)\ra
\\=& 
\dot{c}_i^n\la \pd_c^2u_{c_i^n}(\cdot-x_i^n),J^{-1}\pd_cu_{c_j^n}(\cdot-x_j^n)\ra
+\dot{c}_j^n\la \pd_cu_{c_i^n}(\cdot-x_i^n),J^{-1}\pd_c^2u_{c_j^n}(\cdot-x_j^n)\ra
\\ &+(c_i^n-\dot{x}_i^n)
\la \pd_x\pd_cu_{c_i^n}(\cdot-x_i^n),J^{-1}\pd_cu_{c_j^n}(\cdot-x_j^n)\ra
\\ &+(c_j^n-\dot{x}_j^n)
\la \pd_cu_{c_i^n}(\cdot-x_i^n),J^{-1}\pd_x\pd_cu_{c_j^n}(\cdot-x_j^n)\ra
\\ & + \la \pd_cu_{c_i^n}(\cdot-x_i^n), 
H''(u_{c_j^n}(\cdot-x_j^n))\pd_cu_{c_j^n}(\cdot-x_j^n)\ra
\\ & -\la H''(u_{c_i^n}(\cdot-x_i^n))\pd_cu_{c_i^n}(\cdot-x_i^n),
\pd_cu_{c_j^n}(\cdot-x_j^n)\ra
\\ & +(c_i^n)^{-1}\la H'(u_{c_i^n}(\cdot-x_i^n)),\pd_cu_{c_j^n}(\cdot-x_j^n)\ra
-(c_j^n)^{-1}\la H'(u_{c_j^n}(\cdot-x_j^n)),\pd_cu_{c_i^n}(\cdot-x_i^n)\ra,
\end{align*}
and
$$\frac{d}{dt}\la \pd_cu_{c_i^n}(\cdot-x_i^n),J^{-1}\pd_cu_{c_j^n}(\cdot-x_j^n)\ra
=O(\eps e^{-k_1\eps|x_i^n-x_j^n|})$$
follows from \eqref{eq:xcupb}, Claim \ref{cl:intsize} and the fact that
$H''(u_c)-I=O(u_c)$.
We can compute other components of $\frac{d}{dt}\mathcal{A}_{i,j}^n$
in the same way.
By \eqref{eq:Aij} and \eqref{eq:daij}, we have
\begin{equation}
  \label{eq:vci1}
\|I_1\|_{X_k^n(t)^*} \lesssim \eps^3e^{-k_1\eps d(t)}\|\dv_k^{m,n}\|_{X_k^n(t)^*}.
\end{equation}
By Lemma \ref{lem:parametersupb}, Proposition \ref{prop:vNupb} and
Claim \ref{cl:ucsize},
\begin{align*}
\eps^{-3}(\|\Delta_{11}\|_{X_k^n(t)}+\|\Delta_{12}\|_{X_k^n(t)})
+\|\Delta_{21}\|_{X_k^n(t)}+\|\Delta_{22}\|_{X_k^n(t)}
\lesssim &  \eps^{\frac32}e^{k_1\eps(x_i-x_k)}e^{-k_1\eps d(t)},
\\ \eps^{-3}(\|\Delta_{11}\|_{X_k^n(t)^*}+\|\Delta_{12}\|_{X_k^n(t)^*})
+\|\Delta_{21}\|_{X_k^n(t)^*}+\|\Delta_{22}\|_{X_k^n(t)^*} \lesssim &
\eps^{\frac32}e^{k_1\eps(x_k-x_i)}e^{-k_1\eps d(t)}.
\end{align*}
Hence it follows that
\begin{equation}
  \label{eq:vxi2}
  \|I_2\|_{X_k^n(t)^*}+\|I_3\|_{X_k^n(t)^*}
\lesssim \eps^3e^{-k_1\eps d(t)}\|\dv_k^{m,n}\|_{X_k^n(t)^*}.
\end{equation}
In view of \eqref{eq:Aij}, we see that the first order terms of $I_4$ cancel
each other out and
\begin{equation}
  \label{eq:vxi3}
  \|I_4\|_{X_k^n(t)^*} \lesssim \eps^3e^{-k_1\eps d(t)}\|\dv_k^{m,n}\|_{X_k^n(t)^*}.
\end{equation}
Lemmas \ref{lem:vNXN} and \ref{lem:cauchyxc} imply
\begin{equation}
  \label{eq:xcII}
  \begin{split}
& \|(H''(U_k^m)-H''(U_k^n))v_k^m\|_{X_k^n(t)^*} \\ \lesssim
& \sum_{i=1}^k(|\dc^{m,n}_i|+\eps^3|\dg^{m,n}_i|) \|v_k^m\|_{X_k^n(t)^*}
\\ \lesssim & \eps^2e^{-\frac32k_1\eps d(t)}\left(
\eps^{\frac32}e^{-\frac12k_1\eps d(n)}+\sum_{i=1}^N\|\dv_i^{m,n}\|_{Y_n(t)\cap Z_{i,n}(t)}\right).
  \end{split}
\end{equation}
By Lemma \ref{lem:vNXN}, \eqref{eq:lkn} and \eqref{eq:difpro},
\begin{align*}
\|III\|_{X_k^n(t)^*} \lesssim
\eps^3e^{-\frac32k_1\eps d(t)}\left(\eps^{\frac32}e^{-\frac12k_1\eps d(n)}
+\sum_{i=1}^N\|\dv_i^{m,n}\|_{Y_n(t)\cap Z_{i,n}(t)}\right).
\end{align*}
\par
By Proposition \ref{prop:vNupb}, Lemma \ref{lem:cauchyxc},  
\eqref{eq:ucdiff} and \eqref{eq:rk1dif}--\eqref{eq:rk3dif},
\begin{align*}
\|R_{k1}^m-R_{k1}^n\|_{X_k^n(t)^*} \lesssim &
\|v_k^n\|_{X_k^n(t)^*}\sum_{i=1}^k \|\dv_i^{m,n}\|_{l^2}
+\|\dv_k^{m,n}\|_{X_k^n(t)^*}\sum_{i=1}^k\|v_i^m\|_{l^2}
\\ &+\|v_k^n\|_{X_k^n(t)^*}\sum_{i, j=1}^k\|v_j^n\|_{l^2}(|\dc^{m,n}_i|+\eps^3|\dg^{m,n}_i|)
\\ \lesssim &
\eps^{\frac32}e^{-\frac32k_1\eps d(t)}\left(\eps^{\frac72}e^{-\frac12k_1\eps d(n)}+
\sum_{i=1}^N\|\dv_i^{m,n}\|_{Y_n(t)\cap Z_{i,n}(t)}\right),
\end{align*}
\begin{align*}
\|R_{k2}^m-R_{k2}^n\|_{X_k^n(t)^*}\lesssim &
\eps^{\frac32}e^{-k_1\eps d(t)}\sum_{i=1}^k(|\dc^{m,n}_i|+\eps^3|\dg^{m,n}_i|)
\\ \lesssim &
\eps^2e^{-\frac32k_1\eps d(t)}\left(\eps^{\frac32}e^{-\frac12k_1\eps d(n)}
+\sum_{i=1}^N\|\dv_i^{m,n}\|_{Y_n(t)\cap Z_{i,n}(t)}\right),
\end{align*}
\begin{align*}
& \|R_{k3}^m-R_{k3}^n\|_{X_k^n(t)^*}\\ \lesssim & \eps^2\|\dw^{m,n}_{k-1}\|_{l^2}
+\|w_{k-1}^n\|_{l^2}\sum_{i=1}^k(|\dc^{m,n}_i|+\eps^3|\dg^{m,n}_i|)
\\ \lesssim &
\eps^2e^{-\frac12k_1\eps d(t)}\left\{\sum_{i=1}^{k-1}\|\dw^{m,n}_i\|_{Y_n(t)}
+e^{-k_1\eps d(t)}\left(\eps^{\frac32}e^{-\frac12k_1\eps d(n)}+
\sum_{i=1}^N\|\dv_i^{m,n}\|_{Y_n(t)\cap Z_{i,n}(t)}\right)\right\}.
\end{align*}
By \eqref{eq:rkn} and Proposition \ref{prop:vNupb},
$$\|e^{-k_1\eps(\cdot-x_k^n(t))}R_k^n\|_{l^1}\lesssim \eps^3e^{-k_1\eps d(t)}.$$
Thus by \eqref{eq:difpro} and the fact that
$\|P_k^mJ\|_{B(X_k^n(t)^*,e^{k_1\eps(\cdot-x_k^n(t))}l^1)}=O(\eps^{\frac32})$,
we have
$$\|(P_k^m-P_k^n)P_k^mJR_k^m\|_{X_k^n(t)^*}
\lesssim  \eps^3e^{-\frac32k_1\eps d(t)}\left(\eps^{\frac32}e^{-\frac12k_1\eps d(n)}
+\sum_{i=1}^N\|\dv_i^{m,n}\|_{Y_n(t)\cap Z_{i,n}(t)}\right).$$
Combining the above with Lemma \ref{lem:linearstability}, we see
that for $t\in[T,n]$,
\begin{equation}
  \label{eq:difvmvnq}
  \begin{split}
& \|Q_k^n(t)\dv_k^{m,n}(t)\|_{X_k^n(t)^*} 
\lesssim  e^{b\eps^3(t-n)}\|Q_k^n(n)\dv_k^{m,n}(n)\|_{X_k^n(n)^*}
\\ \qquad & + 
\left(\eps^{\frac92}e^{-\frac12k_1\eps d(n)}+
\eps^3\sum_{i=1}^{k-1}\|\dv_i^{m,n}\|_{Y_n(t)}+\eps^{\frac52}e^{-k_1\eps d(t)}
\sum_{i=1}^N\|\dv_i^{m,n}\|_{Y_n(t)\cap Z_{i,n}(t)}\right)
\\ &\quad \times
\int_t^n e^{b\eps^3(t-s)}\left(1+(\eps^3(s-t)\right)^{-\frac12}
e^{-\frac12k_1\eps d(s)}ds.
  \end{split}
\end{equation}
Thus we have Lemma \ref{lem:vkxkdif} from \eqref{eq:difvmvnp} and
\eqref{eq:difvmvnq}.
\end{proof}

Secondly, we will prove that $\{v_k^n\}_{n=1}^\infty$ is a Cauchy sequence
sequence in the energy space.
\begin{lemma}
  \label{lem:cauchy-en}
There exist an $n_0\in\N$ and a $C>0$ such that for any $m\ge n\ge n_0$
and  $t\in [T,n]$,
\begin{align*}
\|\dv_k^{m,n}\|_{Y_n(t)}^2 \le C\left(
\eps^3e^{-k_1\eps d(n)}+\sum_{i=1}^{k-1}\|\dv_i^{m,n}\|_{Y_n(t)}^2
+\eps^{-1}e^{-k_1\eps d(t)}\sum_{i=1}^N\|\dv_i^{m,n}\|_{Y_n(t)}^2\right).
\end{align*}
\end{lemma}
\begin{proof}
Let
\begin{align*}
& \delta^2H=H'(U_k^m+w_k^m)-H'(U_k^n+w_k^n)-H''(U_k^n+w_k^n)(\dU_k^{m,n}+\dw^{m,n}_k),\\
& \delta N=\tilde{l}_k^m-\tilde{l}_k^n
+\sum_{i=1}^k(l_i^m-l_i^n)-\sum_{i=1}^k(P_i^mJR_i^m-P_i^nJR_i^n).
\end{align*}
Then
\begin{equation*}
\pd_t(\dU_k^{m,n}+\dw_k^{m,n})= JH''(U_k^n+w_k^n)(\dU_k^{m,n}+\dw^{m,n}_k)+
J\delta^2H+\delta N,
\end{equation*}
and we have
\begin{align*}
& \frac{d}{dt}\la H''(U_k^n)(\dU_k^{m,n}+\dw^{m,n}_k), \dU_k^{m,n}+\dw^{m,n}_k\ra
\\ =&
2\la H''(U_k^n)(\dU_k^{m,n}+\dw^{m,n}_k), JH''(U_k^n+w_k^n)(\dU_k^{m,n}+\dw^{m,n}_k)\ra
\\ &
+2\la H''(U_k^m)(\dU_k^{m,n}+\dw^{m,n}_k), J\delta^2H\ra
+2\la H''(U_k^n)(\dU_k^{m,n}+\dw^{m,n}_k), \delta N\ra
\\ &+
\left\la \left(\frac{d}{dt}H''(U_k^n)\right)(\dU_k^{m,n}+\dw^{m,n}_k),\dU_k^{m,n}+\dw^{m,n}_k\right\ra
\\=:& I+II+III+IV.
\end{align*}
Using the skew-adjointness of $J$ and Proposition \ref{prop:vNupb}, we have
\begin{equation}
  \label{eq:lemen1}
\begin{split}
|I| \le & 2\|H''(U_k^n+w_k^n)-H''(U_k^n)\|_{B(l^2)}\|\dU_k^{m,n}+\dw^{m,n}_k\|_{l^2}^2
\\ \lesssim & \eps^{\frac32}e^{-k_1\eps d(t)} \|\dU_k^{m,n}+\dw^{m,n}_k\|_{l^2}^2.
\end{split}
\end{equation}
Proposition \ref{prop:vNupb} and Corollary \ref{cor:ucdiff} imply
\begin{equation}
  \label{eq:lemen2}
|II|\lesssim \|\dU_k^{m,n}+\dw^{m,n}_k\|_{l^2}^3 \lesssim 
\eps^{\frac32}e^{-k_1\eps d(t)} \|\dU_k^{m,n}+\dw^{m,n}_k\|_{l^2}^2.  
\end{equation}
\par
Next, we will estimate $III$.
In view of Lemma \ref{lem:cauchyxc} and \eqref{eq:dotx-cdif}, we have
\begin{align*}
&  e^{\frac12k_1\eps d(t)}(|\pd_t\dc^{m,n}_k|+\eps^3|\pd_t\dg^{m,n}_k-\dc^{m,n}_k|)
\\ \lesssim & \eps^{\frac72}e^{-\frac12k_1\eps d(t)}\left\{
\sum_{i=1}^{k-1}\|\dv_i^{m,n}\|_{Y_n(t)}
+e^{-k_1\eps d(t)}\left(\eps^{\frac32}e^{-\frac12k_1\eps d(n)}
+\sum_{i=1}^N \|\dv_i^{m,n}\|_{Y_n(t)\cap Z_{i,n}(t)}\right)\right\}.
\end{align*}
Combining the above with \eqref{eq:modeq1'}, \eqref{eq:lemxc4} and
\eqref{eq:lemxc5}, we obtain
$$|\alpha_{ik}^m-\alpha_{ik}^n|+\eps^3|\beta_{ik}^m-\beta_{ik}^n|
\lesssim \eps^{\frac72}e^{-\frac32k_1\eps d(t)}\left(\eps^{\frac32}e^{-\frac12k_1\eps d(n)}+
\sum_{i=1}^N \|\dv_i^{m,n}\|_{Y_n(t)\cap Z_{i,n}(t)}\right).$$
Thus we have 
\begin{align*}
& e^{\frac12k_1\eps d(t)}\left(\|\tilde{l}_k^m-\tilde{l}_k^n\|_{W_k^n(t)^*}
+\sum_{i=1}^k\|l_i^m-l_i^n\|_{W_k^n(t)^*}\right)
\\ \lesssim & \eps^3\sum_{i=1}^{k-1}\|\dv_i^{m,n}\|_{Y_n(t)}
+\eps^3e^{-k_1\eps d(t)}\left(\eps^{\frac32}e^{-\frac12k_1\eps d(n)}
+\sum_{i=1}^N\|\dv_i^{m,n}\|_{Y_n(t)\cap Z_{i,n}(t)}\right).
\end{align*}
\par

By Proposition \ref{prop:vNupb} and \eqref{eq:rk1dif}--\eqref{eq:rk3dif}, 
\begin{align*}
  & \|R_{k1}^m-R_{k1}^n\|_{\widetilde{W}_k(t)} \lesssim 
\eps^{\frac32}e^{-k_1\eps d(t)}(\|\dw^{m,n}_{k-1}\|_{W_k^n(t)}+\|\dv_k^{m,n}\|_{W_k^n(t)}),
\\ & \|R_{k2}^m-R_{k2}^n\|_{W_k^n(t)} \lesssim 
\eps^2e^{-\frac32k_1\eps d(t)} \left(\eps^{\frac32}e^{-\frac12k_1\eps d(n)}
+\sum_{i=1}^N \|\dv_i^{m,n}\|_{Y_n(t)\cap Z_{i,n}(t)}\right),
\end{align*}
and
\begin{align*}
& \|R_{k3}^m-R_{k3}^n\|_{W_k^n(t)} \\ \lesssim & 
\eps^2e^{-\frac12k_1\eps d(t)}\left\{\sum_{i=1}^{k-1}\|\dv_i^{m,n}\|_{Y_n(t)}
+e^{-k_1\eps d(t)}\left(\eps^{\frac32}e^{-\frac12k_1\eps d(n)}
+\sum_{i=1}^N \|\dv_i^{m,n}\|_{Y_n(t)\cap Z_{i,n}(t)}\right)\right\}.
\end{align*}
in the same way as the proof of Lemma \ref{lem:vkxkdif}.
Since $\|R_k^n\|_{\widetilde{W}_k(t)^*}\lesssim \eps^3e^{-k_1\eps d(t)}$ and
\begin{align*}
& \|P_i^nJ\|_{B(W_k^n(t),W_k^n(t)^*)}+\eps^{-\frac12}
\|P_i^nJ\|_{B(\widetilde{W}_k^n(t),W_k^n(t)^*)}=O(\eps),
\\ &  \eps^{\frac12}\|(P_i^m-P_i^n)J\|_{B(W_k^n(t),W_k^n(t)^*)}
+\|(P_i^m-P_i^n)J\|_{B(\widetilde{W}_k^n(t),W_k^n(t)^*)}
\\ \lesssim & e^{-\frac12 k_1\eps d(t)}\left(\eps^{\frac32}e^{-\frac12k_1\eps d(n)}
+\sum_{k=1}^N \|\dv_k^{m,n}\|_{Y_n(t)\cap Z_{k,n}(t)}\right),
\end{align*}
it follows that
\begin{align*}
&  \|P_i^mJR_i^m-P_i^nJR_i^n\|_{W_k^n(t)^*}\\ \le &
\|(P_i^m-P_i^n)JR_i^m\|_{W_k^n(t)^*}+\|P_i^nJ(R_i^m-R_i^n)\|_{W_k^n(t)^*}
\\ \lesssim & 
\eps^3e^{-\frac12k_1\eps d(t)} \left\{\sum_{i=1}^{k-1}\|\dv_i^{m,n}\|_{Y_n(t)}
+e^{-k_1\eps d(t)}\left(\eps^{\frac32}e^{-\frac12k_1\eps d(n)}
+\sum_{i=1}^N \|\dv_i^{m,n}\|_{Y_n(t)\cap Z_{i,n}(t)}\right)\right\}.
\end{align*}
Thus we have
\begin{equation}
  \label{eq:lemen3}
  \begin{split}
|III|\lesssim & \|\dU_k^{m,n}+\dw^{m,n}_k\|_{W_k^n(t)}\|\delta N\|_{W_k^n(t)^*}
\\ \lesssim &
\eps^3\|\dU_k^{m,n}+\dw^{m,n}_k\|_{W_k^n(t)}^2
+\eps^3e^{-k_1\eps d(t)}\sum_{i=1}^{k-1}\|\dv_i^{m,n}\|_{Y_n(t)}^2
\\ &
+\eps^3e^{-3k_1\eps d(t)}\left(\eps^3e^{-k_1\eps d(n)}+
\sum_{i=1}^N \|\dv_i^{m,n}\|_{Y_n(t)\cap Z_{i,n}(t)}^2\right). 
  \end{split}
\end{equation}

Since $\|\frac{d}{dt}H''(U_k^n)\|_{B(W_k^n(t),W_k^n(t)^*)}=O(\eps^3)$,
\begin{equation}
  \label{eq:lemen7}
  |IV|\lesssim  \eps^3\|\dU_k^{m,n}+\dw^{m,n}_k\|_{W_k^n(t)}^2.
\end{equation}
By Lemma \ref{lem:vkxkdif}, Corollary \ref{cor:ucdiff}
and the fact that $\|v_k(t)\|_{W_k^n(t)}\lesssim \|v_k(t)\|_{X_k^n(t)^*}$,
\begin{equation}
  \label{eq:lemen4}
  \begin{split}
& \|\dU_k^{m,n}+\dw^{m,n}_k\|_{W_k^n(t)}^2 \\ \lesssim & \sum_{i=1}^{k-1}\|\dv_i^{m,n}\|_{l^2}^2
+\|\dv_k^{m,n}\|_{X_k^n(t)^*}^2 +\|\dU_k^{m,n}\|_{l^2}^2
\\ \lesssim & \eps^3e^{-k_1\eps(d(t)+d(n))}
+\sum_{i=1}^{k-1}e^{-k_1\eps d(t)}\|\dv_i^{m,n}\|_{Y_n(t)}^2
+ \eps^{-1}e^{-3k_1\eps d(t)}\sum_{i=1}^N \|\dv_i^{m,n}\|_{Y_n(t)}^2.
  \end{split}
\end{equation}
Combining \eqref{eq:lemen1}--\eqref{eq:lemen4}, we obtain
\begin{equation}
  \label{eq:lemen8}
  \begin{split}
    & \left|\frac{d}{dt}\la H''(U_k^n)(\dU_k^{m,n}+\dw^{m,n}_k), \dU_k^{m,n}+\dw^{m,n}_k\ra\right|
\\ \lesssim & \eps^{\frac32}e^{-k_1\eps d(t)}\|\dU_k^{m,n}+\dw^{m,n}_k\|_{l^2}^2
+\eps^6e^{-k_1\eps(d(t)+d(n))}+\eps^3e^{-k_1\eps d(t)}\sum_{i=1}^{k-1}\|\dv_i^{m,n}(t)\|_{Y_n(t)}^2
\\ & +\eps^2e^{-3k_1\eps d(t)}\sum_{i=1}^N\|\dv_i^{m,n}\|_{Y_n(t)}^2.
  \end{split}
\end{equation}
Integrating \eqref{eq:lemen8}over $[t,n]$, we have
\begin{equation}
  \label{eq:dukwkl2}
  \begin{split}
& \left(1+O(\eps^{-\frac32}e^{-k_1\eps d(t)})\right)
\|\dU_k^{m,n}(t)+\dw^{m,n}_k(t)\|_{l^2}^2
\\ \lesssim & \eps^3e^{-k_1\eps(d(n)+d(t))}
+e^{-k_1\eps d(t)}\sum_{i=1}^{k-1}\|\dv_i^{m,n}(t)\|_{Y_n(t)}^2
+\eps^{-1}e^{-3k_1\eps d(t)}\sum_{i=1}^N\|\dv_i^{m,n}\|_{Y_n(t)}^2.
  \end{split}
\end{equation}
Combining the above with Corollary \ref{cor:ucdiff}, we have
Lemma \ref{lem:cauchy-en}. Thus we complete the proof.
\end{proof}
Now we are in position to prove existence of $N$-soliton states.
\begin{proof}[Proof of the former part of Theorem \ref{thm:1}]
By \ref{lem:cauchyxc}, \ref{lem:vkxkdif} and \ref{lem:cauchy-en},
there exist positive constant $C$ and $n_0\in\N$ such that
\begin{equation}
\label{eq:1f1}
\sum_{k=1}^N e^{\frac12 k_1\eps d(t)}\left\{\eps^{\frac12}\|\dv_k^{m,n}(t)\|_{l^2}
+|\dc_k^{m,n}(t)|+\eps^3|\dg_k^{m,n}(t)|\right\} \le Ce^{-\frac12k_1\eps d(n)}
\end{equation}
for every $m\ge n\ge n_0$ and $t\in[T,n]$. Therefore 
$$v_k(t)=\lim_{n\to\infty}v_k^n(t), \quad
c_k(t)=\lim_{n\to\infty}c_k^n(t),\quad
\gamma_k(t)=\lim_{n\to\infty}\gamma_k^n(t)
$$
exist for every $t\ge T$ and it follows from Proposition \ref{prop:vNupb} that
\begin{equation}
  \label{eq:1f2}
 \eps^{\frac12}\|v_k(t)\|_{l^2}+|c_k(t)-c_{k,+}|+\eps^3|\gamma_k(t)-\gamma_{k,+}|
\le A\eps^2e^{-k_1\eps d(t)}
\end{equation}
for every $1\le k\le N$ and $t\ge T$, where $A$ is a constant independent of $t$.
Moreover, 
$$u(t,n):=\sum_{k=1}^N\left(u_{c_k(t)}(n-c_{k,+}t-\gamma_k(t))+v_k(t,n)\right)
$$
is a solution of \eqref{eq:FPU}. In view of \eqref{eq:1f2},
\begin{align*}
& \left\|u(t)-\sum_{k=1}^Nu_{c_{k,+}}(\cdot-c_{k,+}t-\gamma_{k,+})\right\|_{l^2}
\\ \le & \sum_{k=1}^N\|v_k(t)\|_{l^2}+
\|u_{c_{k(t)}}(\cdot-c_{k,+}t-\gamma_k(t))-
u_{c_{k,+}}(\cdot-c_{k,+}t-\gamma_{k,+})\|_{l^2}
\\ \lesssim & \eps^{\frac32}e^{-k_1\eps d(t)}.
\end{align*}
Thus we prove existence of a solution to \eqref{eq:FPU} satisfying
\eqref{eq:convergence}.
\end{proof}

\section{Uniqueness}
\label{sec:uniqueness}
In this section, we will prove uniqueness of a solution of \eqref{eq:FPU}
that converges to a sum of $N$ solitary waves as $t\to\infty$.
First, we will prove that an asymptotic $N$-soliton state converges to
a sum of $N$ solitary waves as $t\to\infty$ if it is sufficiently small.
\begin{proposition}
  \label{prop:expconv}
Let $k_{N}>\cdots>k_1>0$, $c_{i,+}=1+\frac{(k_i\eps)^2}{6}$ and $\gamma_{i,+}\in\R$
for $1\le i\le N$. There exists a positive number $\eps_0$ such that if
$\eps\in(0,\eps_0)$ and $u(t)$ is a solution of \eqref{eq:FPU} satisfying
\eqref{eq:convergence}, then \eqref{eq:expconvergence} holds for a $\beta>0$.
\end{proposition}
To prove Proposition \ref{prop:expconv}, we will decompose $u(t)$
as in Section \ref{subsec:2.1}.
Let $$U_{N,+}(t)=\sum_{i=1}^Nu_{c_{i,+}}(\cdot-c_{i,+}t-\gamma_{i,+})$$ and let
$v_{0n}(t)$ and $v_{kn}(t)$ $(1\le k\le N)$ be solutions of
\begin{equation}
  \label{eq:v0n}
  \left\{
    \begin{aligned}
& \pd_tv_0^n=JH'(v_0^n),\\
& v_{0n}(n)=u(n)-U_{N,+}(n),
    \end{aligned}\right.
\end{equation}
and
\begin{equation}
  \label{eq:vkn}\left\{
  \begin{aligned}
&  \pd_tv_k^n=JH''(U_k^n(t))v_k^n+l_k^n+Q_k^n(t)JR_k^n,\\
& v_k^n(n)=0,
  \end{aligned}\right.
\end{equation}
where $w_k^n=\sum_{i=0}^kv_i^n$ $(0\le k\le N)$, $U_k^n$, $R_k^n$ and $l_k^n$
are defined as \eqref{eq:Ukdef}, \eqref{eq:Rkdef} and \eqref{eq:lkdef} and
$x_i^n(t)$, $c_i^n(t)$, $\alpha_{ik}^n(t)$, $\beta_{ik}^n(t)$ are
solutions of \eqref{eq:xcinit}--\eqref{eq:orthvk3}.
Then as in Section \ref{subsec:2.1}, we have $u(t)=U_N^n(t)+w_N^n(t)$
for every $n\in\N$. Moreover, following the proof of \cite[Lemmas 4.1 and 4.2]
{Mi2}, we see that there exist $n_0\in N$ and $T>0$ such that
the decomposition above exists for every $t\in [T,n]$ and $n\ge n_0$.
\begin{lemma}
  \label{lem:modulation-u}
Assume that $c_{i,+}$ and $T$ be as in Theorem \ref{thm:1}.
Let $u(t)$  be a solution of \eqref{eq:FPU} satisfying \eqref{eq:convergence}
and let $v_0^n$ and $v_i^n$ $(1\le i\le N)$ be solutions of \eqref{eq:v0n} and 
\eqref{eq:vkn}, respectively. Suppose that $x_i^n$ and $c_i^n$
$(1\le i\le N)$ are solutions of  \eqref{eq:xcinit} and \eqref{eq:modeq1}
and that $\alpha_{ij}^n$ and $\beta_{ij}^n$ ($1\le i\le k\le N$) are
$C^1$-functions satisfying \eqref{eq:orthvk3}.
Then $u=U_N^n+w_N^n$ and 
$v_k^n$ satisfies \eqref{eq:orthv2k} for $1\le i\le k\le N$.
\par
Furthermore, there exist a positive constant $A$ and  an $n_0\in\N$ such that
for  $n\ge n_0$, $t\in[T,n]$ and $1\le k\le N$,
\begin{gather*}
|\dot{c}_k^n(t)|+\eps^3|\dot{x}_k^n(t)-c_k^n(t)| \le A\eps^{\frac72}
\left(\sum_{i=0}^N\|v_i^n(t)\|_{W^n(t)}+\eps^{\frac32}e^{-k_1\eps d(t)}\right),\\
|\ddot{c}_k^n(t)|+\eps^3|\ddot{x}_k^n(t)|\le A\eps^{\frac92}
\left(\sum_{i=0}^N\|v_i^n(t)\|_{W^n(t)}+\eps^{\frac32}e^{-k_1\eps d(t)}\right),
\end{gather*}
and 
\begin{equation}
\label{eq:c-u}
  \begin{split}
 & \left|\frac{d}{dt}\left\{c_k^n(t)\left(1-\theta_1(c_k^n(t))^{-1}
\la w_{k-1}^n(t),\rho_{c_k^n(t)}\ra\right)\right\}\right|
\\ \le  &
A\eps^2\left(\|v_0^n(t)\|_{W^n(t)}^2+\sum_{i=1}^N\|v_i^n(t)\|_{W(t)\cap X_k(t)}^2
+\eps^3e^{-2k_1\eps d(t)}\right),
  \end{split}
\end{equation}
where $\rho_c=\pd_x(c\pd_x+J)^{-1}(H'(u_c)-u_c)$.
\end{lemma}
\begin{proof}
Except for \eqref{eq:c-u}, Lemma \ref{lem:modulation-u} can be shown in exactly
the same way as Lemmas \ref{lem:modulation} and \ref{lem:parametersupb}.
\par
Since \eqref{eq:c-u} can be shown in the same way as Lemma 2.2 in \cite{Mi2},
we only give an outline of the proof.
As in the proofs of \cite[Lemmas 2.2 and 2.5]{Mi2} and  \eqref{eq:ham28},
 we have
\begin{align*}
& \left|\theta_1(c_i^n)\dot{c}_i^n-c_i^n\la R_{i3}^n,\pd_xu_{c_i^n}\ra\right|
\lesssim 
\eps^3\left(\|v_0^n\|_{W^n(t)}^2+\sum_{k=1}^N\|v_k^n\|_{X_k^n(t)^*\cap W^n(t)}^2
+\eps^3e^{-k_1\eps d(t)}\right),\\
& R_{k3}^n= (H''(u_{c_k^n})-I)w_{k-1}^n+\sum_{\substack{1\le i,\,j\le k\\ i\ne j}}
O(|w_{k-1}^n|(|u_{c_i}^n||u_{c_j}^n|+|w_{k-1}^n|)),
\end{align*}
and
\begin{align*}
& \left| \la w_{k-1}^n,(H''(u_{c_k}^n)-I)\pd_xu_{c_k}^n\ra
+\frac{d}{dt}\la w_{k-1}^n,\rho_{c_k}^n(\cdot-x_k^n)\ra\right|
\\ \lesssim &
\eps^3
\left(\|v_0^n\|_{W^n(t)}+\sum_{k=1}^N\|v_k^n\|_{X_k^n(t)^*\cap W^n(t)}\right)^2
+\eps^6e^{-k_1\eps d(t)}.
\end{align*}
Combining the above, we obtain \eqref{eq:c-u}.
\end{proof}

The following energy estimates of $v_k^n(t)$ can be shown in the 
same way as  Lemma \ref{lem:speed-Hamiltonian}.
\begin{lemma}
  \label{lem:sp-Ham-u}
Let $u(t)$ and $v_i^n(t)$ $(0\le i\le N)$ be as in Lemma \ref{lem:modulation-u}
and let $c_{i,+}$ and $T$ be as in Theorem \ref{thm:1}.
Then there exist an $n_0\in \N$ and a $C>0$ such that for $n\ge n_0$
and $t\in[T,n]$,
$$\|v^n(t)\|_{l^2}^2\le C\left(\eps\sum_{i=1}^N|c_i^n(t)-c_{i,+}|+
\eps^{\frac32}\sum_{i=0}^N \|v_i^n\|_{W^n(t)}+\eps^3e^{-k_1\eps d(t)}\right),$$
\begin{align*}
\|v_k^n(t)\|_{l^2}^2 \le & C\left(\eps\sum_{i=1}^k|c_i^n(t)-c_{i,+}|
+\eps^{\frac32}\sum_{i=0}^k\|v_i^n\|_{W_k^n(t)}\right)
\\ & +C\eps^3\left(\sum_{i=1}^N\|v_i^n\|_{L^2(t,n;W^n(s)\cap X_i^n(s)^*)}^2
+e^{-k_1\eps d(t)}\right)\quad\text{for $1\le k\le N$.}
\end{align*}
\end{lemma}
We have the following local energy estimates for large $n$.
\begin{lemma}
  \label{lem:virial-0u}
Let $v_0^n(t)$ be a solution of \eqref{eq:v0n}. Then there exists a positive
constant $C$ such that for every $n\in\N$,
\begin{equation*}
\sup_{t\in\R}\|v_0^n(t)\|_{l^2}+\eps^{\frac32}\|v_0^n\|_{L^2(-\infty,n;W^n(t))}
\le C\|v_0^n(n)\|_{l^2}.
\end{equation*}
\end{lemma}

\begin{lemma}
  \label{lem:locenergy-ku}
Let $u(t)$ and $v_i^n(t)$ $(0\le i\le N)$ be as in Lemma \ref{lem:modulation-u}
and let $c_{i,+}$ and $T$ be as in Theorem \ref{thm:1}.
Then there exist $n_0\in\N$ and a positive constant $C$ such that 
for $n\ge n_0$, $t_1\in[T,n]$ and $1\le k\le N$,
\begin{align*}
&\la\psi_{k_1\eps,N}w_k^n(t_1),w_k^n(t_1)\ra_{l^2}
+ \eps^{\frac32}\|w_k^n\|_{L^2(t_1,n;W^n(t))}
\\ \le & C\eps^{\frac32}\left(\|v_k^n\|_{L^2(t_1,n;X_k^n(t)^*)}
+\sum_{i=0}^{k-1}\|w_i^n\|_{L^2(t_1,n;W_k^n(t))}+e^{-k_1\eps d(t_1)}\right),
\end{align*} 
$$\|v_k^n(t_1)\|_{X_k^n(t_1)}+\eps^{\frac32}\|v_k^n\|_{L^2(t_1,n;X_k^n(t))}
\le C\eps^{\frac32}\left(\|w_{k-1}^n\|_{L^2(t_1,n;W_k^n(t))}+e^{-k_1\eps d(t_1)}\right).$$
\end{lemma}
Since Lemmas \ref{lem:virial-0u} and  \ref{lem:locenergy-ku} can be proved
in the same way as \cite[Lemma 3.2]{Mi2} and Lemmas \ref{lem:virialv2k} and
\ref{lem:vNXN}, we omit the proof.
Now we are in position to prove Proposition \ref{prop:expconv}.
\begin{proof}[Proof of Proposition \ref{prop:expconv}]
Lemmas \ref{lem:modulation-u} and \ref{lem:virial-0u} imply that
for $1\le k\le N$ and $t\in[T,n]$,
\begin{equation}
\label{eq:vknw}
\|v_k^n(t)\|_{W^n(t)\cap X_k^n(t)^*}
+\eps^{\frac32}\|v_k^n(t)\|_{L^2(t,n;W^n(s)\cap X_k^n(s)^*)}
\le C(\|v_0^n(n)\|_{l^2}+\eps^{\frac32}e^{-k_1\eps d(t)}),
\end{equation}
where $C$ is a positive constant independent of $n\ge n_0$.
It follows from \eqref{eq:c-u}, Claim \ref{cl:rhoc} and the above that
\begin{equation}
  \label{eq:cinb}
  \begin{split}
\eps^{-2}|c_i^n(t)-c_{i,+}|\lesssim &
\eps^{-\frac32}(\|v_0^n(n)\|_{l^2}+\|w_{i-1}^n(t)\|_{W^n(t)})
\\ & +\sum_{i=0}^N\|v_i^n\|_{L^2(t,n;W^n(s))}^2
+\sum_{i=1}^N\|v_i^n\|_{L^2(t,n;X_i^n(s)^*)}^2+e^{-k_1\eps d(t)}
\\ \lesssim & \eps^{-\frac32}(\|v_0^n(n)\|_{l^2}+\eps^{\frac32}e^{-k_1\eps d(t)})  
  \end{split}
\end{equation}
for every $t\in[T,n]$ and $n\ge n_0$ and it follows from
Lemma \ref{lem:sp-Ham-u} that
\begin{equation}
  \label{eq:vknb}
\|v_k^n(t)\|_{l^2}^2\lesssim  \eps^{\frac32}
(\|v_0^n(n)\|_{l^2}+\eps^{\frac32}e^{-k_1\eps d(t)}).
\end{equation}
Moreover, Lemma \ref{lem:modulation-u} and \eqref{eq:vknw} imply that
\begin{gather}
\label{eq:xc-n}
|\dot{c}_i^n(t)|+\eps^3|\dot{x}_i^n(t)-c_i^n(t)|\lesssim
\eps^{\frac72}(\|v_0^n(n)\|_{l^2}+\eps^{\frac32}e^{-k_1\eps d(t)}),
\\ \label{eq:xc-n2}
|\ddot{c}_i^n(t)|+\eps^3|\ddot{x}_i^n(t)|\lesssim
\eps^{\frac92}(\|v_0^n(n)\|_{l^2}+\eps^{\frac32}e^{-k_1\eps d(t)}). 
\end{gather}
\par

Let $I$ be a compact subinterval of $[T,\infty)$.
By \eqref{eq:cinb}--\eqref{eq:xc-n2} and the assumption \eqref{eq:convergence},
we see that $x_k^n(t)$ and $c_k^n(t)$ are bounded in $C^2(I)$.
Thus by Arzela's theorem, there exit subsequences
$\{c_k^{n_j}(t)\}_{j=1}^\infty$, $\{x_k^{n_j}(t)\}_{j=1}^\infty$ and
$C^1$-functions $c_k(t)$ and $c_k(t)$ such that for any compact interval
$I\subset [T,\infty)$,
$$\lim_{j\to\infty}\|c_k^{n_j}(t)-c_k(t)\|_{C^1(I)}
=\lim_{j\to\infty}\|x_k^{n_j}(t)-x_k(t)\|_{C^1(I)}=0.$$
Since $\lim_{n\to\infty}\|v_0^n(n)\|_{l^2}=0$, it follows from \eqref{eq:vknb}
and \eqref{eq:xc-n} that for $1\le k\le N$ and $t\ge T$,
\begin{gather*}
\limsup_{n\to\infty}\sum_{k=0}^N\|v_k^n(t)\|_{l^2}\lesssim\eps^{\frac32}e^{-k_1\eps d(t)},
\quad 
|\dot{c}_k(t)|+\eps^3|\dot{x}_k(t)-c_k(t)|\lesssim \eps^5e^{-k_1\eps d(t)}.
\end{gather*}
Thus there exist real constants $\tilde{c}_{k,+}$ and $\tilde{\gamma}_{k,+}$
such that
\begin{equation}
\label{eq:convrate}
|c_k(t)-\tilde{c}_{k,+}|+\eps^3|x_k(t)-\tilde{c}_{k,+}t-\tilde{\gamma}_{k,+}|
\lesssim \eps^2e^{-k_1\eps d(t)}.
\end{equation}
Let $\widetilde{U}_{N,+}(t)=\sum_{k=1}^Nu_{\tilde{c}_{k,+}}
(\cdot-\tilde{c}_{k,+}t-\tilde{\gamma}_{k,+})$.
By \eqref{eq:convrate},
\begin{align*}
&  \|u(t)-\widetilde{U}_{N,+}(t)\|_{l^2}
\\ \lesssim & \lim_{j\to\infty}\sum_{k=1}^N\|u_{c_k^{n_j}(t)}(\cdot-x_k^{n_j}(t))
-u_{\tilde{c}_{k,+}}(\cdot-\tilde{c}_{k,+}t-\tilde{\gamma}_{k,+})\|_{l^2}
+\limsup_{j\to\infty}\sum_{k=0}^N\|v_k^{n_j}(t)\|_{l^2}
\\ \lesssim &
\eps^{-\frac12}\sum_{k=1}^N\left(|c_k(t)-\tilde{c}_{k,+}|
+\eps^3|x_k(t)-\tilde{c}_{k,+}t-\tilde{\gamma}_{k,+}|\right)
+\eps^{\frac32}e^{-k_1\eps d(t)}
\\ \lesssim & \eps^{\frac32}e^{-k_1\eps d(t)}.
\end{align*}
Combining \eqref{eq:convergence} and the above, we have
\begin{align*}
&  \|U_{N,+}(t)-\widetilde{U}_{N,+}(t)\|_{W(t)}
\\ \le &
 \|u(t)-U_{N,+}(t)\|_{l^2}+\|u(t)-\widetilde{U}_{N,+}(t)\|_{W(t)}\to0
\quad\text{as $t\to\infty$,}
\end{align*}
and $\widetilde{U}_{N,+}(t)=U_{N,+}(t)$.
This completes the proof of Proposition \ref{prop:expconv}.
\end{proof}

Finally, we will prove the uniqueness of solutions to \eqref{eq:FPU}
that satisfy \eqref{eq:convergence}.
\begin{proof}[Proof of of the latter part of Theorem \ref{thm:1}]
Let $u(t)$ and $\bar{u}(t)$ be solutions of \eqref{eq:FPU} satisfying
\eqref{eq:convergence}. By Proposition \ref{prop:expconv},
$$\|u(t)-U_{N,+}(t)\|_{l^2}+\|\bar{u}(t)-U_{N,+}(t)\|_{l^2}
=O(\eps^{\frac32}e^{-k_1\eps d(t)}).$$
\par
First we decompose $\bar{u}(t)$ in the same way as $u(t)$. 
Let
$$\bar{u}(t,\cdot)=\sum_{k=1}^Nu_{\bar{c}_k^n(t)}(\cdot-\bar{x}_k^n(t))
+\sum_{k=0}^N\bar{v}_k^n(t,\cdot),$$
where $\bar{v}_{0n}(t)$ and $\bar{v}_{kn}(t)$ $(1\le k\le N)$ are solutions of
$$\pd_t\bar{v}_0^n=JH'(\bar{v}_0^n),\qquad
\bar{v}_{0n}(n)=\bar{u}(n)-U_{N,+}(n),$$
$$\pd_t\bar{v}_k^n=JH''(\overline{U}_k^n(t))\bar{v}_k^n+\bar{l}_k^n
+\overline{Q}_k^n(t)J\overline{R}_k^n,\qquad \bar{v}_k^n(n)=0,$$
$\bar{w}_k^n=\sum_{i=0}^k\bar{v}_i^n$ for $k\ge 0$, $\overline{U}_0^n(t)=0$,
$\overline{U}_k^n(t)=\sum_{i=1}^ku_{\bar{c}_k^n(t)}(\cdot-\bar{x}_k^n(t))$
for $k\ge1$,
\begin{gather*}
\overline{R}_k^n(t)=H'(\overline{U}_k^n+\bar{w}_k^n)
-H'(\overline{U}_{k-1}^n+\bar{w}_{k-1}^n)-H'(u_{\bar{c}_k^n}(\cdot-\bar{x}_k^n))
-H''(\overline{U}_k^n)\bar{v}_k^n,\\
\bar{l}_k^n(t)=\sum_{i=1}^k\left(\bar{\alpha}_{ik}^n(t)
\pd_cu_{\bar{c}_i^n(t)}(\cdot-\bar{x}_k^n(t))
+\bar{\beta}_{ik}^n(t)\pd_xu_{\bar{c}_i^n(t)}(\cdot-\bar{x}_k^n(t))\right),
\end{gather*}
and  $\bar{x}_i^n(t)$, $\bar{c}_i^n(t)$, $\bar{\alpha}_{ik}^n(t)$ and
$\bar{\beta}_{ik}^n(t)$ are chosen in the same way as 
$x_i^n(t)$, $c_i^n(t)$, $\alpha_{ik}^n(t)$ and $\beta_{ik}^n(t)$.
Then $\bar{u}(t)=\overline{U}_N^n(t)+\bar{w}_N^n(t)$,
$\bar{x}_k^n(n)=c_{k,+}n+\gamma_{k,+}$, $\bar{c}_k^n(n)=c_{k,+}$,
\begin{gather*}
\la \bar{v}_k^n(t),J^{-1}\pd_xu_{\bar{c}_i^n(t)}(\cdot-\bar{x}_i^n(t))\ra
=\la \bar{v}_k^n(t),J^{-1}\pd_cu_{\bar{c}_i^n(t)}(\cdot-\bar{x}_i^n(t))\ra=0  
\end{gather*}
for $1\le i\le k\le N$ and $t\in[T,n]$.
\par

Let $\dc_k^n=c_k^n-\bar{c}_k^n$, $\dg_k^n=x_k^n-\bar{x}_k^n$,
$\dv_k^n=v_k^n-\bar{v}_k^n$ and $\dw_k^n=w_k^n-\bar{w}_k^n$.
Proposition \ref{prop:expconv} and Lemma \ref{lem:virial-0u} imply that
there exist a $C>0$ and an $n_0\in\N$ such that for $t\in[T,n]$ and $n\ge n_0$,
$|\dg_k^n(t)|\le C\eps^{-1}e^{-k_1\eps d(t)}$ and
\begin{equation}
  \label{eq:bv0n}
 \sup_t(\|v_0^n(t)\|_{l^2}+\|\bar{v}_0^n(t)\|_{l^2})
\lesssim \|v_0^n(n)\|_{l^2}+\|\bar{v}_0^n(n)\|_{l^2}\le C\eps^{\frac32}e^{-k_1\eps d(n)}.
\end{equation}
Following the proof of Lemmas  \ref{lem:cauchyxc}, \ref{lem:vkxkdif}
and \ref{lem:cauchy-en} and making use of \eqref{eq:bv0n},
we see that there exists a positive constant $C$ such that
for $n\ge n_0$ and $t\in[T,n]$,
\begin{align*}
& \sup_{s\in[t,n]}e^{\frac12k_1\eps d(s)}(|\dc_k^n(s)|+\eps^3|\dg_k^n(s)|)
\\ \le  & C\eps^2e^{-\frac12k_1\eps d(n)}+
C\eps^{\frac12}\left(\sum_{i=1}^{k-1}\|\dv_i^n\|_{Y_n(t)}
+e^{-k_1\eps d(t)}\sum_{i=1}^N\|\dv_i^n\|_{Y_n(t)\cap Z_{i,n}(t)}\right),
\end{align*}
\begin{align*}
 \|\dv_k^n\|_{Z_{k,n}(t)} \le & C\eps^{\frac32}e^{-\frac12k_1\eps d(n)}
+C\left(\sum_{i=1}^{k-1}\|\dv_i^n\|_{Y_n(t)}
+\eps^{-\frac12}e^{-k_1\eps d(t)}\sum_{i=1}^{N}\|\dv_i^n\|_{Y_n(t)}\right),
\end{align*}
and
\begin{equation*}
\|\dv_k^n\|_{Y_n(t)}^2 \le C\left(\eps^3e^{-k_1\eps d(n)}
+\sum_{i=1}^{k-1}\|\dv_i^n\|_{Y_n(t)}^2
+\eps^{-1}e^{-k_1\eps d(t)}\sum_{i=1}^N\|\dv_i^n\|_{Y_n(t)}^2\right).
\end{equation*}
Combining the above, we have for every $t\in[T,n]$ and $n\ge n_0$,
$$|\dc_k^n(t)|+\eps^3|\dg_k^n(t)|+\eps^{\frac12}\|\dv_k^n(t)\|_{l^2}
\lesssim \eps^2e^{-\frac12k_1\eps(d(t)+d(n))}.$$
Therefore
\begin{align*}
& \sup_{s\in[t,n]} \|u(s)-\bar{u}(s)\|_{l^2}
\\ \le & 
\sup_{s\in[t,n]}\left(\sum_{k=0}^N\|\dv_k^n(s)\|_{l^2}+\sum_{k=1}^N
\|u_{c_k^n(s)}(\cdot-x_k^n(s))-u_{\bar{c}_k^n(s)}(\cdot-\bar{x}_k^n(s))\|_{l^2}\right)
\\ \le & C\eps^{\frac32}e^{-\frac12k_1\eps d(n)},
\end{align*}
where $C$ is a positive constant independent of $n\ge n_0$.
Letting $n\to\infty$, we have $$\sup_{t\ge T}\|u(t)-\bar{u}(t)\|_{l^2}=0.$$
Thus we complete the proof of Theorem \ref{thm:1}.
\end{proof}
\section*{Acknowledgment}
This research is supported by Grant-in-Aid for Scientific
Research (No. 21540220).

\appendix
\section{Size of $u_c$ and $\rho_c$}
\label{sec:size}
In this section, we recollect estimates on the size of solitary waves $u_c$
and the size of interaction between solitary waves for the sake of
self-containedness.
See \cite[Appendix A]{Mi2} for the proof.
\begin{claim}
  \label{cl:ucsize}
Let $c=1+\frac{1}{6}\eps^2$, $a\in(\frac14\eps,\frac74\eps)$ and
let $i$ and $j$ be nonnegative integers. Then
\begin{align*}
& \|\pd_x^i\pd_c^ju_c\|_{l^2_a\cap l^2_{-a}}=O(\eps^{\frac32+i-2j}),
\quad \|J^{-1}\pd_x^i\pd_c^ju_c\|_{l^2_{-a}}=O(\eps^{\frac12+i-2j}),
\\ &
\|\pd_x^i\pd_c^ju_c\|_{l^\infty_{_a}\cap l^\infty_{-a}}=O(\eps^{2+i-2j}),
\quad \|J^{-1}\pd_x^i\pd_c^ju_c\|_{l^\infty\cap l^\infty_{-a}}=O(\eps^{1+i-2j}).
\end{align*}
\end{claim}
\begin{claim}
  \label{cl:intsize}
Let $0<k_1<k_2$ and $a\in[0,\frac74\eps)$.
Then there exists an $\eps_*>0$ such
that if $\eps\in(0,\eps_*)$ and $c_i=1+\frac{k_i^2\eps^2}6$ for $i=1$, $2$,
\begin{align*}
& \|\pd_x^{\alpha_1}\pd_c^{\beta_1}u_{c_1}(\cdot-x_1)
\pd_x^{\alpha_2}\pd_c^{\beta_2}u_{c_1}(\cdot-x_2)\|_{l^\infty}
=O(\eps^{4+\alpha_1+\alpha_2-2(\beta_1+\beta_2)}
e^{-k_1a|x_2(t)-x_1(t)|}),
\\ &
\|\pd_x^{\alpha_1}\pd_c^{\beta_1}u_{c_1}(\cdot-x_1)
\pd_x^{\alpha_2}\pd_c^{\beta_2}u_{c_1}(\cdot-x_2)\|_{l^1}
=O(\eps^{3+\alpha_1+\alpha_2-2(\beta_1+\beta_2)}e^{-k_1a|x_2(t)-x_1(t)|}).
\end{align*}
\end{claim}
 \begin{claim}
   \label{cl:4}
Let $a_1,\cdots,a_N\in\R$ and
$I=\{\sum_{i=1}^N\theta_ia_i:
 0\le \theta_i\le 1\text{ for $1\le i\le N$}\}$.
Suppose $f\in C^2(\R)$ and $f(0)=0$. Then
$$\left|f(\sum_{1\le i\le N}a_i)-\sum_{1\le i\le N}f(a_i)\right|
\le \sup_{x\in I}|f''(x)|\sum_{i\ne j}|a_ia_j|.$$
 \end{claim}
\begin{claim}
Let $a\in[0,2k_1\eps)$. Then
  \label{cl:rhoc}
\begin{gather*}
\|\pd_x^i\pd_c^j\rho_c\|_{l^2_a\cap l^2_{-a}}
+\|J^i\pd_c^j\rho_c\|_{l^2_a\cap l^2_{-a}}
=O(\eps^{\frac32+i-2j}),\\
\|\pd_x^i\pd_c^j\rho_c\|_{l^2_a\cap l^2_{-a}}
+\|J^i\pd_c^j\rho_c\|_{l^\infty_a\cap l^\infty_{-a}}
=O(\eps^{2+i-2j}).
\end{gather*}
\end{claim}
\begin{claim}
  \label{cl:j-1}
Let $a>0$ and $u=(u_1,u_2)\in l^2_{-a}$, $v=(v_1,v_2)\in l^2_a$ and 
$J^{-1}$ be an inverse operator of $J$ defined as \eqref{eq:J-1}.
Then 
$$ \la u, J^{-1}v\ra=
-\la u_1, \sum_{k=1}^\infty e^{k\pd}v_2\ra
-\la \sum_{k=1}^\infty e^{k\pd}u_2, v_1\ra,$$
and as $l\to\infty$,
\begin{gather*}
\la u,J^{-1}e^{l\pd}v\ra
=O(a^{-1}e^{-al}\|u\|_{l^2_{-a}}\|v\|_{l^2_a}),
\\ \la u,J^{-1}e^{-l\pd}v\ra=-\la u_1,1\ra\la v_2,1\ra
-\la u_2,1\ra\la v_1,1\ra+O(a^{-1}e^{-al}\|u\|_{l^2_{-a}}\|v\|_{l^2_a}).
\end{gather*}
Furthermore,  $\la u,J^{-1}u\ra=-\la u_1,1\ra\la u_2,1\ra$
if $u\in l^2_a\cap l^2_{-a}$.
\end{claim}

\end{document}